\documentclass[11pt]{article}
\usepackage{mathtools,color}

\usepackage{fixltx2e,indentfirst,xspace,bm}
\usepackage[retainorgcmds]{IEEEtrantools}
\usepackage{framed,paralist}

\usepackage{amssymb}
\usepackage{amsmath}
\usepackage{graphicx}
\usepackage{amsthm}
\usepackage{cite}
\usepackage[margin=1in]{geometry}
\usepackage{cite}
\usepackage{appendix}
\usepackage[colorlinks]{hyperref}
\definecolor{link1}{rgb}{0,0,.7}
\definecolor{link2}{rgb}{0,0.25,0.5}

\newtheorem{theorem}{Theorem}[section]
\newtheorem{lemma}{Lemma}[section]
\newtheorem{proposition}{Proposition}[section]
\newtheorem{corollary}{Corollary}[section]

\theoremstyle{remark}
\newtheorem{remark}{Remark}[section]

\theoremstyle{definition}
\newtheorem{definition}{Definition}[section]
\allowdisplaybreaks[4]

\numberwithin{equation}{section}

\def\br{\begin{remark}}
\def\er{\end{remark}}
\def\bp{\begin{proposition}}
\def\ep{\end{proposition}}
\def\bc{\begin{corollary}}
\def\ec{\end{corollary}}
\def\bd{\begin{definition}}
\def\ed{\end{definition}}

\def\non{\nonumber }

\def\INT{\int_{R_0}^{R_1}}

\newcommand{\RR}{\mathbb R}
\newcommand{\NN}{\mathbb N}
\newcommand{\mcS}{\mathcal S}

\DeclareMathOperator{\tr}{tr}
\newcommand\defeq{\stackrel{\scriptscriptstyle \text{def}}=}

\newcommand\del{\partial}

\newcommand\eps{{\varepsilon}}

\newcommand{\R}{{\mathbb R}}

\newcommand{\F}{\mathcal F}
\newcommand{\E}{\mathcal E}
\newcommand{\Id}{\mathbb I}

\newcommand{\bulk}{_\textit{bulk}}
\newcommand{\el}{_\textit{el}}

\renewcommand{\ge}{\geqslant}
\renewcommand{\le}{\leqslant}
\renewcommand{\geq}{\geqslant}
\renewcommand{\leq}{\leqslant}
\DeclarePairedDelimiter{\norm}{\lVert}{\rVert}
\DeclarePairedDelimiter{\abs}{\lvert}{\rvert}

\makeatletter
\newcommand{\step}[1]{\par\emph{#1\@addpunct}}
\makeatother

\begin{document}
\title{
  Dynamic Cubic Instability in a 2D\\
   Q-tensor Model for Liquid Crystals
  \begingroup
  \footnotetext{%
    This work was partially supported by NSF grants
    DMS-0806703, 
    DMS-1007914, 
    and
    DMS-1252912. 
    GI also acknowledges partial support from an Alfred P. Sloan research fellowship.
    The authors thank the Center for Nonlinear Analysis (NSF grants DMS-0405343 and DMS-0635983),
     PIRE grant OISE-0967140, and OxPDE, the Oxford Centre for Nonlinear PDE, where part of this research was carried out.
  }
  \endgroup
}
\author{
  Gautam Iyer%
  \footnote{
    Dept.\ of Math.\ Sci.,
    Carnegie Mellon University, USA.
    \texttt{gautam@math.cmu.edu}}
\and
  Xiang Xu\footnote{
  Dept.\ of Math.\ Sci.,
  Purdue University, USA.
  \texttt{xuxiang@andrew.cmu.edu}}
\and
  Arghir D. Zarnescu\footnote{
  University of Sussex, UK
  and 
 Institute of Mathematics ``Simion Stoilow" of the Romanian Academy, 
  Romania.
  \texttt{A.Zarnescu@sussex.ac.uk}}
}
\date{}
\maketitle
\begin{abstract}
We consider a four-elastic-constant Landau-de Gennes energy
characterizing nematic liquid crystal configurations described using
the $Q$-tensor formalism. The energy contains a cubic term and is
unbounded from below. We study dynamical effects produced by the
presence of this cubic term by considering an $L^2$ gradient flow
generated by this energy. We work in two dimensions and  concentrate
on understanding the relations between the physicality of the
initial data and the global well-posedness of the system.
\end{abstract}

\section{Introduction}
This paper studies the dynamics of an important instability
phenomenon that arises in the Landau-de Gennes theory of nematic
liquid crystals \cite{B12, BM10, BZ11}. Mathematically our results
address global well-posedness of the $L^2$ gradient flow generated
by an energy functional that is unbounded from below in its natural
energy space. This turns out to be related to quantifying how the
flow affects the convex hull of the initial data.

We consider a Landau-de Gennes energy functional
$$
  \E
   [Q]=\int_{\Omega}\F (Q(x))\,dx,
$$
 where $\Omega\subset\RR^d$ with
$d=2,3$ and  $Q$ is a matrix valued function defined on $\Omega$
that takes values into the space of $Q$-tensors, namely $\mcS^{(d)}
\defeq \{M\in\RR^{d\times d}, \,M=M^{T},\,\tr(M)=0\}$. The matrix
$Q(x)$ is a measure of the local preferred orientation of the
nematic molecules at the point $x\in\Omega$, see for instance
\cite{NM04, BZ11}.

The energy density  $\F(Q)$ can be decomposed as:
$$
   \F(Q)=\F\el+\F\bulk
$$
 where $\F\el$ is the ``elastic part'' which depends on gradients of $Q$,
 and $\F\bulk$ is the ``bulk part'' that contains no gradients.

Invariances under physical symmetries impose certain restrictions on the form
of the elastic and bulk parts. The simplest and most common form
that is invariant under physical symmetries and still captures the
essential features~\cite{B12, NM04} assumes that $\mathcal F\el$ and
$\mathcal F\bulk$ are given by
\begin{align}\label{DefElasticPart}
 \mathcal{F}\el(Q)&\defeq L_1|\nabla{Q}|^2+L_2\partial_jQ_{ik}\partial_kQ_{ij}
+L_3\partial_jQ_{ij}\partial_kQ_{ik}
+L_4Q_{lk}\partial_kQ_{ij}\partial_lQ_{ij},\\
\label{DefBulkPart} \mathcal{F}\bulk(Q)&\defeq
\frac{a}{2}\tr(Q^2)-\frac{b}{3}\tr(Q^3)+\frac{c}{4}\tr^2(Q^2).
\end{align}
Here and in the following we assume the Einstein summation
convention by which repeated indices $i,j,k=1,\dots,d$ are
implicitly summed.

The coefficients $a,b,c$ and $L_k,k=1,2,3,4$ are assumed to be
non-dimensional (see~\cite{MG00}). For spatially homogeneous systems
the term $\mathcal{F}\bulk$ is bounded from below only if $c> 0$
(see~\cite{PZ12}). Physical considerations impose that $b\ge 0$
(see~\cite{M10}) and $a$ is a temperature dependent parameter that
can be taken to be either positive or negative. The most physically
relevant case is when $a$ is small. This corresponds to a
temperature near the supercooling point, below which the isotropic
phase becomes unstable. Thus we make the assumptions
\begin{equation}\label{AssumptionsBulk}
 b\ge 0
 \quad\text{and}\quad
 c>0.
\end{equation}
In two dimensions observe that $Q\in\mcS^{(2)}$ implies
$\tr(Q^3)=0$. Hence we may, without loss of generality, assume
$b=0$.

For the elastic part we note that the first three terms are
quadratic, while the fourth one (with coefficient $L_4$) is cubic.
The presence of a cubic term is rather unusual in most physical
systems. The retention of this term in our situation is motivated by the fact that it
allows reduction of the elastic energy $\mathcal{F}[Q]$ to the
classical Oseen-Frank energy of liquid crystals
 (with four elastic terms).
This is done by formally taking
$$
  Q(x)=s_+\left(n(x)\times n(x)-\frac{1}{d}\Id\right)
  \quad\text{where }
  s_+>0, \quad
  n:\Omega\to\mathbb{S}^{d-1},
$$
and substituting it in the definition of $\mathcal{E}[Q]$ (see Appendix~\ref{OF_LD} or~\cite{BZ11}).
Here $\Id$ denotes the identity matrix.

The cubic term, however, also comes with a price: The energy
$\mathcal E[Q]$ now has the ``unpleasant'' feature of being
unbounded from below~\cite{B12, BM10}. On the other hand, if $L_4=0$
the elastic part of $\mathcal{E}[Q]$,
$$
\mathcal E\el[Q] \defeq \int_\Omega\mathcal{F}\el(Q(x))\,dx,
$$
is bounded from below (and coercive) if and only if $L_1, L_2$ and $L_3$ satisfy certain conditions.
For $Q\in\mcS^{(3)}$ and three dimensional domains these conditions are developed in~\cite{LMT87} (see also \cite{TE98}).
For $Q\in\mcS^{(2)}$ and two dimensional
domains the conditions
\begin{equation}\label{eqn2DCoercivity}
L_1+L_2>0 \quad\text{and}\quad L_1+L_3>0,
\end{equation}
are equivalent to coercivity.
(We prove this in Lemma~\ref{lmaPositiveEnergy} in Appendix~\ref{sxnCoercivity}.)

One way to deal with the unboundedness and lack of coercivity caused
by the (necessary) presence of $L_4$ is to
 replace the bulk potential defined in \eqref{DefBulkPart}  with a potential $\psi(Q)$, which is finite if and only if $Q$ is
\emph{physical}%
\footnote{%
We recall~\cite{BM10, M10} that $Q$ is \emph{physical} if
$Q\in\mcS^{(d)}$ and after suitable non-dimensionalisations its
eigenvalues are between $-\frac{1}{d}$ and $1-\frac{1}{d}$. } (see
for instance~\cite{BM10} for $d = 3$). In this paper we aim to
directly study the physical relevance of the energy $\mathcal{E}[Q]$
keeping the more common potential \eqref{DefBulkPart}, instead of
the singular potential as in~\cite{BM10} (see \cite{FRSZ13, MW12}
for works in the dynamical context). Of course the static theory
will not provide anything meaningful when the energy $\mathcal E[Q]$
is unbounded. Consequently, we focus our attention on the dynamical
aspect.

We study a gradient flow in the ``simplest setting''; namely an
$L^2$ gradient flow in $\mathbb{R}^2$ corresponding to the energy
functional $\mathcal{E}[Q]$ where $Q$ takes values in $\mcS^{(2)}$.
Explicitly, this is
\begin{equation}\label{EquationQAbstract}
\frac{\partial Q_{ij}}{\partial
t}=-\left(\frac{\delta\mathcal{E}}{\delta Q}\right)_{ij}+\lambda
\delta_{ij} +\mu_{ij}-\mu_{ji}
\end{equation}
where  $\lambda$ is a Lagrange  multiplier corresponding to the
constraint $\tr(Q)=0$ and for $i, j \in \{1,2\}$ the
$\mu_{ij}$'s are the Lagrange multipliers  corresponding to the
constraints $Q_{ij}=Q_{ji}$. Here $\frac{\delta \E}{\delta Q}$
denotes the variational derivative of $\E$ with respect to $Q$,
defined by
\begin{equation*}
   \frac{\delta \E}{\delta Q}(\varphi) = \left. \frac{d}{dt} \E(Q + t
    \varphi) \right|_{t=0}
\end{equation*}
for $\varphi \in C^\infty_c(\Omega, M^{d\times d}(\mathbb{R}))$.
Integrating by parts as necessary we can identify the linear
operator $\frac{\delta \E }{\delta Q}$ with a matrix-valued
function.

After some lengthy but straightforward calculations (which we carry
out in Appendix~\ref{sec:gradient flow derivation})
equation~\eqref{EquationQAbstract} reduces to
\begin{align}\label{eqnQExpansion}
\frac{\partial Q_{ij}}{\partial t} =2L_1&\Delta{Q_{ij}}-aQ_{ij}-c
\,\tr(Q^2)Q_{ij}
+(L_2+L_3)\left(\partial_j\partial_kQ_{ik}+\partial_i\partial_kQ_{jk}\right)\non\\
&-(L_2+L_3)\partial_l\partial_kQ_{lk}\delta_{ij}
+2L_4\partial_lQ_{ij}\partial_kQ_{lk}+2L_4\partial_l\partial_kQ_{ij}Q_{lk}
\non\\
&-L_4\partial_iQ_{kl}\partial_jQ_{kl}+\frac{L_4}{2}|\nabla
Q|^2\delta_{ij}.
\end{align}
We study this system of equations on a bounded domain
$\Omega\subset\mathbb{R}^{2}$ with  initial data and boundary conditions
given by
\begin{equation}\label{eqnQICBC}
Q(x, 0)=Q_0(x),
\quad\text{and}\quad
Q(x, t)|_{\partial\Omega}=\tilde{Q}(x).
\end{equation}
\begin{equation*}
  Q_0|_{\partial\Omega}=\tilde Q.
\end{equation*}

The main results in this paper are to show:
\begin{compactitem}
  \item Global existence of weak solutions to~\eqref{eqnQExpansion}--\eqref{eqnQICBC} in two dimensions,
  for $H^1\cap L^\infty$ initial data that is small in $L^\infty$ (Theorem~\ref{thm2Dgexist}, below).

  \item Finite time blow up (in $L^2$) of solutions to~\eqref{eqnQExpansion}--\eqref{eqnQICBC}
  in two dimensions, for specially constructed (large) initial data (Theorem~\ref{thmBlowUp}, below).

  \item The ``preservation of physicality'' of the initial data in two or three dimensions and a
  simple version of the flow (Proposition~\ref{prop:physicality}, below).
\end{compactitem}
We defer the precise statements (and proofs) of these results to subsequent sections,
and momentarily pause to briefly outline the ideas involved in the proofs and the problems encountered.

The main difficulty in proving global existence stems from the fact that the energy is apriori unbounded from below.
 However, from equation~\eqref{eqnQExpansion} we see that if $\|Q\|_{L^\infty}$ is small enough,
 then the cubic term can be absorbed into the other terms, which are positive definite under the assumption~\eqref{eqn2DCoercivity}.
 Here
 \begin{equation*}
  \norm{Q}_{L^\infty} = \sup_{x\in\Omega} |Q(x)|,
  \quad\text{ where }
  |Q(x)|^2 = \tr\left(Q(x)Q(x)^t\right) = \tr\left(Q^2(x)\right).
 \end{equation*}
 Thus the usual $H^1$-level information provided by the energy in such gradient flows can be effectively utilized,
 provided we apriori guarantee a smallness condition on the $L^\infty$-norm.
 Our main tool (Proposition~\ref{first proposition on 2D small data}) does precisely this: namely,
 Proposition~\ref{first proposition on 2D small data} shows smallness of~$\norm{Q}_{L^\infty}$ globally in time, provided it is small enough initially.
 We use this to prove global existence of weak solutions in Theorem~\ref{thm2Dgexist}.
 Global existence of strong solutions should now follow using relative standard methods, provided
 the initial data is regular, small and is compatible with the boundary conditions (see for instance \cite{EV98}).

 We complement Theorem~\ref{thm2Dgexist} with Theorem~\ref{thmBlowUp} which shows the existence of a finite time
 blow up using large, specially constructed initial data.
  The proof amounts to finding a non-linear differential inequality for a quantity that blows up in finite time.
  The main difficulty in this context is again the high order nonlinearity.
 We use the energy inequality for control of this, even though the sign of the energy is not apriori controlled.

 Theorems~\ref{thm2Dgexist} and~\ref{thmBlowUp} give a dichotomy common to many nonlinear PDE's:
  long time existence if the initial data is small enough, and examples of finite time blow-up for large data.
  This leads to an interesting question about the maximal size of initial data for which solutions exist globally in time.
  This is a very subtle one and we only provide a modest contribution in this direction.
  We think that an important factor affecting global existence is the \emph{physicality} of the initial data
  -- namely the requirement that after a particular normalization the eigenvalues of the initial data are within
  the interval $(-\frac{1}{d},1-\frac{1}{d})$ (see more about physicality in \cite{B12, BM10}).

There exists a direct and delicate relation between the smallness of $\|Q\|_{L^\infty}$ and the aforementioned notion of ``physicality''.
Specifically, the physicality of a $Q$-tensor imposes an upper bound on the size of $\|Q\|_{L^\infty}$ but in general the contrary is false.
Namely having an upper bound for $\|Q\|_{L^\infty}$ implies physicality in 2D, but not necessarily in higher dimensions.

 More precisely, if $Q\in\mcS^{(d)}$ is physical, i.e. its eigenvalues $\lambda_i,i=1,\dots,d$ are
  in the interval $(-\frac{1}{d},1-\frac{1}{d})$, hence $\tr(Q^2(x))=\sum_{i=1}^d \lambda_i^2\le d(1-\frac{1}{d})^2$.
  On the other hand, the condition $\tr(Q^2(x))=\sum_{i=1}^d \lambda_i^2\le d(1-\frac{1}{d})^2$ for $Q\in\mcS^{(d)}$
   implies that the eigenvalues of $Q$ are between $(-\frac{1}{d},1-\frac{1}{d})$ only for $d=2$, but not for $d=3$!
   For $d=3$, the notion of physicality is related to $Q$ belonging to a convex set (not just a ball as for  $d=2$).
   Proposition~\ref{prop:physicality} explores how the gradient flow preserves the convex hull of the initial data in a simple setting, for both $d=2$ and $d=3$.

\subsection*{Plan of this paper.}

This paper is organized as follows.
In section~\ref{sxnMain} we precisely state the main results of this paper and state our notational conventions.
In section~\ref{sec:well-posedness} prove the small data global existence result (Theorem~\ref{thm2Dgexist}).
In section~\ref{sec:blow-up} we exhibit an example of a finite time blow up with large initial data.
In section~\ref{sec:physicality} we prove the preservation of physicality (Proposition~\ref{prop:physicality}).

There are numerous technical calculations involved in this paper, which for clarity of presentation have been relegated to appendices.
Appendix~\ref{sec:gradient flow derivation} shows that the gradient flow defined by~\eqref{EquationQAbstract} satisfies~\eqref{eqnQExpansion}.
Appendix~\ref{OF_LD} shows how the Landau-de~Gennes energy functional can be reduced to the Oseen-Frank energy
functional in two dimensions, and the necessity of the cubic term for this purpose.
Appendix~\ref{sxnCoercivity} shows that the coercivity assumption~\ref{eqn2DCoercivity} is equivalent to coercivity in two dimensions.
Finally Appendix~\ref{sxnCalculations} reduces the evolution for $Q$ into a one dimensional problem when the
 initial data is of the type used to prove the blow up in Theorem~\ref{thmBlowUp}.

\section{Main results and notational conventions.}\label{sxnMain}

Our first main result in this paper is global well-posedness of~\eqref{eqnQExpansion} for small initial data.
The crucial step in the proof is the preservation of $L^\infty$-smallness, and we begin by stating this.
\begin{proposition}\label{first proposition on 2D small data}
Consider the $2D$ evolution problem  \eqref{eqnQExpansion}-\eqref{eqnQICBC} on a bounded smooth domain
$\Omega\subset\RR^2$. Suppose the coercivity condition
\eqref{eqn2DCoercivity} holds together with the structural assumptions \eqref{AssumptionsBulk}
. For smooth solutions $Q$ there exists an explicitly
computable constant $\eta_1$ (depending on $L_i,i=1,\dots,4$) so
that if
\begin{equation}\label{boundary:assumption}
 \|\tilde{Q}\|_{L^\infty(\partial\Omega)} \le \|Q_0\|_{L^\infty(\Omega)}<\sqrt{2\eta_1}
\end{equation}
and
\begin{equation}\label{small coefficient assumption}
|a| \leq 2c\eta_1,
\end{equation}
then for any $T>0$, we have
\begin{equation}\label{eqnQBound}
 \|Q\|_{L^\infty((0, T)\times\Omega)} \leq \sqrt{2\eta_1}.
\end{equation}
\end{proposition}
\begin{remark}
As mentioned earlier, the physically relevant regime is when the
parameter $a$ has small magnitude. This is consistent with the
assumption \eqref{small coefficient assumption}.

 Furthermore a careful check of the proof of Proposition~\ref{first proposition on 2D small data}
  shows that~\eqref{eqnQBound} still holds for weak solutions that satisfy~\eqref{boundary:assumption} and~\eqref{small coefficient assumption}.
\end{remark}
\begin{theorem}\label{thm2Dgexist}
Suppose the coefficients $a, b, c$ and $L_1$, \dots, $L_4$ satisfy
the coercivity condition~\eqref{eqn2DCoercivity} together with the
structural assumptions~\eqref{AssumptionsBulk}, and let
$\Omega\subset\RR^2$ be a smooth, bounded domain. There exists an
explicitly computable constant $\eta_2$ (depending on
$L_i,i=1,\dots,4$ and $\Omega$) so that if $Q_0 \in H^1(\Omega)\cap
L^\infty(\Omega)$, $\tilde{Q}\in H^\frac32(\partial\Omega)$, and the
smallness conditions \eqref{boundary:assumption} and \eqref{small
coefficient assumption} hold with $\eta_1$ replaced by $\eta_2$,
then the system~\eqref{eqnQExpansion}--\eqref{eqnQICBC} has a unique global weak solution%
\footnote{see Definition~\ref{def of weak solution} for the precise
definition of a weak solution}. Further the initial
smallness~\eqref{boundary:assumption} is preserved for all time.
\end{theorem}

We prove Proposition~\ref{first proposition on 2D small data} and Theorem~\ref{thm2Dgexist} in section~\ref{sec:well-posedness}.
The smallness assumption on the initial data is essential;
we complement Theorem~\ref{thm2Dgexist} with a result showing that certain solutions exhibit a finite time blow up.

\begin{theorem}\label{thmBlowUp}
Suppose the coefficients $a, b, c$ and $L_1$, \dots, $L_4$ satisfy
the coercivity condition~\eqref{eqn2DCoercivity} together with the
structural assumptions~\eqref{AssumptionsBulk}. There exists a
smooth domain $\Omega$, smooth initial data $Q_0$, and a smooth
(time independent) function $\tilde Q: \del \Omega \to \R$  such
that the system~\eqref{eqnQExpansion} with Dirichlet boundary
conditions $\tilde Q$ does not admit a global smooth solution.
\end{theorem}
\begin{remark}
  Our proof (Section~\ref{sec:blow-up}) chooses $\Omega$ to be the annulus
   $B_{R_1}(0)\setminus B_{R_0}(0) \subset \R^2$ where $0 < R_0 < R_1$, and ``hedgehog'' type initial data.
  Namely, we choose $Q_0$ of the form
  \begin{equation*}
    Q(0) = \theta_0(|x|) \left(\frac{x}{|x|}\otimes \frac{x}{|x|}-\Id_2\right),
  \end{equation*}
  where $\theta_0: [R_0, R_1] \to \R$ is smooth.
  If $\theta_0$ is large enough, and $R_0$, $R_1$ are such that
  \begin{equation*}
    \frac{R_0^2\pi^2}{9(R_1 - R_0)^2}>1,
  \end{equation*}
  we show $\norm{Q(t)}_{L^2(\Omega)} \to \infty$ in finite time, for any smooth solution.
\end{remark}

Finally in Section~\ref{sec:physicality} we study how the flow distorts the convex hull
of eigenvalues, in an attempt to understand what is the maximal size of initial data that would give global well-posedness.
The situation is more interesting in $3D$ than in $2D$ as in $3D$ the convex set of physical $Q$-tensors cannot
be described just in terms of the Frobenius norm of the matrix.
We restrict ourselves to a simple setting (with specific assumptions on the elastic constants $L_i$'s, $i=1,2,3,4$  and work in the whole space).
Our main result in this section is the following:
\begin{proposition}\label{prop:physicality}
Let $Q(t,x)\in C([0,T];H^k(\RR^d))$ with $k>\frac{d}{2}$, $d=2,3$ and
arbitrary $T>0$ be a solution of the system \eqref{eqnQExpansion}--\eqref{eqnQICBC}, under assumptions \eqref{AssumptionsBulk}.
Assume further
\begin{itemize}
\item $L_1\not=0$, $L_4=0$ and \eqref{eqn2DCoercivity} holds if $d = 2$,
\item or $L_1\neq0$ and   $L_2+L_3=L_4=0$ if $d = 3$.
\end{itemize}
Suppose the initial data $Q_0\in H^k(\RR^d)$ is such that for any $x \in \RR^d$, the eigenvalues of $Q_0(x)$ are in the interval
\begin{gather*}
\Big[-\sqrt{\frac{|a|}{2c}},\sqrt{\frac{|a|}{2c}}\Big]
\quad\text{when } d=2,\\
\llap{or\qquad}
\Big[-\frac{b+\sqrt{b^2-24ac}}{12c},\frac{b+\sqrt{b^2-24ac}}{6c}\Big]
\quad\text{when } d= 3.
\end{gather*}
If $d=3$, we further assume
\begin{equation}\label{restriction:a}
|a|<\frac{b^2}{3c}.
\end{equation}
Then, for any $t\in [0,T]$ and $x\in\RR^d$, the eigenvalues of $Q(t,x)$ stay in the same interval.
\end{proposition}

The ``usual'' energy methods do not seem to yield Proposition~\ref{prop:physicality} in dimension $d=3$.
Instead we use a Trotter product formula and provide a somewhat atypical proof in section~\ref{sec:physicality}.

\subsection*{Notational Convention.}
We define $A:B \defeq \tr(A^tB)$ when $A,B$ are $d\times d$
matrices, and let $|Q|$  denote the Frobenius norm of the matrix $Q$
(i.e. $|Q| \defeq \sqrt{\tr (Q^tQ)}=\sqrt{\tr (Q^2)}$). We denote
the space of $Q$-tensors by $\mcS^{(d)}$, where
\begin{equation*}
  \mcS^{(d)} \defeq \{M\in\RR^{d\times d}, \,M=M^t,\,\tr(M)=0\},
\end{equation*}
and define the matrix valued $L^p$ space by
\begin{equation*}
  L^p(\Omega,\mcS^{(d)}) \defeq \{Q:\Omega\to\mcS^{(d)}, |Q|\in L^p(\Omega,\RR)\},\quad\text{when } 1 \leq p \leq \infty.
\end{equation*}
For the sake of simplicity, we let $\|\cdot\|$ (with no subscripts) to denote
$\|\cdot\|_{L^2(\Omega)}$.
We denote the partial derivative with
respect to $x_k$ of the $ij$ component of $Q$, by either $Q_{ij,k}$
or $\partial_k Q_{ij}$. Throughout the paper, we assume the Einstein summation
convention over the repeated indices.

\section{Global well-posedness for small initial data}\label{sec:well-posedness}
Using standard techniques the gradient flow structure of the
equation should provide apriori estimates for~\eqref{EquationQAbstract}
for smooth enough solutions. Taking the
(matrix) inner product of equation~\eqref{EquationQAbstract} with
$\frac{\delta \E}{\delta Q} - \lambda{I}+\mu-\mu^T$ and integrating
yields
$$
\frac{d}{dt} \E[Q] = -\int_{\Omega} \left| \frac{\delta \E}{\delta
Q}-\lambda{I}+\mu-\mu^T \right|^2\,dx.
$$
This gives the energy equality
\begin{equation}\label{eqnEE}
\E[Q(t)] + \int_0^t \int_\Omega  \left|\frac{\delta \E}{\delta
Q}-\lambda{I}+\mu-\mu^T  \right| ^2 \, dx \, ds
 = \E[Q(0)], \ \ \ \ \ \ \ \forall t>0.
\end{equation}

The main defect of the energy $\E[Q]$ is that it is unbounded from
below as $L_4\not=0$. Thus, unlike in the usual contexts, it does
not provide apriori control over the $H^1$ norm of $Q$. On the other
hand, if $\norm{Q}_{L^\infty}$ is small enough, then we can absorb
the cubic term into the three quadratic terms and force the elastic
part of the energy to be positive. The idea behind our proof is to
first prove \emph{preservation of smallness}: namely, if
$\norm{Q}_{L^\infty}$ is small enough initially, then it does not
increase with time. Now coercivity of the quadratic terms, and
smallness of the cubic term force the energy $\E[Q]$ to stay
positive, from which~\eqref{eqnEE} will provide an a~priori $H^1$
bound for $Q$. This will be enough to prove well-posedness of
\eqref{EquationQAbstract} (or equivalently
equation~\eqref{eqnQExpansion}).

\subsection{Preservation of smallness in \texorpdfstring{$L^\infty$}{L-infinity}}

The goal of this section is to prove Proposition~\ref{first
proposition on 2D small data} showing that $L^\infty$ smallness of
the initial data is preserved in time. This in turn implies that the
energy is positive definite and will allow us to obtain apriori
estimates on higher norms.

We begin by recalling a few well-known results that come directly
from Gagliardo-Nirenberg inequalities and elliptic PDE theory.
\begin{lemma} \label{def of C1 and C2} Suppose $\Omega$ is a smooth, bounded
domain in $\mathbb{R}^2$. There exists a positive constant $C_1 =
C_1(\Omega)$, such that for any $f\in H^2(\Omega)$ and $g\in
H^{\frac32}(\partial\Omega)$, with $f|_{\partial\Omega}=g$, we have
\begin{equation}\label{rel:estc12}
\|f\|_{L^\infty(\Omega)}\leq
C_1\|f\|^{\frac12}\big(\|\Delta{f}\|^{\frac12}+\|f\|^{\frac12}+
\|g\|_{H^{\frac32}}^{\frac12}\big), \ \ \|D^2f\| \leq
C_1\big(\|\Delta{f}\|+\|f\|+\|g\|_{H^{\frac32}(\partial\Omega)}
\big).
\end{equation}
Moreover, for any $f\in H^2(\Omega)$, we have the interpolation
estimate
\begin{equation} \label{rel:estc42}\|\nabla f\|_{L^4}^2\le
C_1\|f\|_{L^\infty}\left(\|\Delta f\|_{L^2}+\|f\|+
\|g\|_{H^{\frac32}}^{\frac12}\right).
\end{equation}
Finally, for $f\in H^1_0(\Omega)$, we have the Ladyzhenskaya
inequality \cite{LZ69}
 \begin{equation}
 \|f\|_{L^4(\Omega)}^2\le C\|\nabla f\|\|f\|.
\end{equation}
  \end{lemma}
\br Further, for $f\in H^2(\Omega)\cap H_0^1(\Omega)$ the terms
$\|f\|$ and $\|g\|_{H^{\frac32}(\partial\Omega)}$ are not required
in \eqref{rel:estc12} and \eqref{rel:estc42}. This follows from
elliptic regularity (see for instance \cite{EV98}, Section $6.3.2$,
Thm. $4$ and remark $(i)$ afterwards).
 \er

The proofs of \eqref{rel:estc12}, follow from interpolation
inequalities (see for instance \cite[Theorems 5.2, 5.8]{AF03})
combined with the elliptic regularity~\cite[Theorem 6.3.2.4]{EV98}.
The estimate \eqref{rel:estc42} is a consequence of
Gagliardo-Nirenberg inequality (see for instance \cite{BR11}, p.313)
combined with the elliptic regularity result previously mentioned.

We can now provide the proof of Proposition~\ref{first proposition
on 2D small data}:

\begin{proof}[Proof of Proposition~\ref{first proposition on 2D small data}]
Due to the special structure of $Q$ in $2D$, we expand $Q$ as
\begin{eqnarray} Q(x, t)=\left(
                \begin{array}{cc}
                  p(x_1,x_2,t) & q(x_1,x_2,t) \\
                  q(x_1,x_2,t) & -p(x_1,x_2,t) \\
                \end{array}
            \right),
               \label{Q in matrix form}
 \end{eqnarray}
where $p, q$ are two scalar functions. Inserting \eqref{Q in matrix
form} into \eqref{eqnQExpansion}, we obtain the following
evolution equations for $p$ and $q$:
 \begin{align}
\frac{\partial{p}}{\partial{t}}=\zeta\Delta{p}&+L_4\big[(\partial_1p)^2-(\partial_1q)^2-(\partial_2p)^2
+(\partial_2q)^2+2\partial_1p\partial_2q+2\partial_2p\partial_1q \big] \non\\
&+2L_4(p\partial_1\partial_1p+2q\partial_1\partial_2p-p\partial_2\partial_2p)-ap-2c(p^2+q^2)p,
\label{p equ} \\
\frac{\partial{q}}{\partial{t}}=\zeta\Delta{q}&+2L_4\big[\partial_1q\partial_2q-\partial_1p\partial_2p+\partial_1p\partial_1q
-\partial_2p\partial_2q \big] \non\\
&+2L_4(p\partial_1\partial_1q+2q\partial_1\partial_2q-p\partial_2\partial_2q)-aq-2c(p^2+q^2)q,
\label{q equ} \\
p(x, 0)=p_0&(x), \ \ q(x)=q_0(x), \ \ p(x,
t)|_{\partial\Omega}=\tilde{p}(x), \ \
q(x,t)|_{\partial\Omega}=\tilde{q}(x).
 \label{eqnPQICBC}
 \end{align}
 Here, $\tilde{p}$ and $\tilde{q}$ are the corresponding components
 associated to $\tilde{Q}$
 and
\begin{equation}
  \zeta\defeq 2L_1+L_2+L_3> 0. \label{zeta}
\end{equation}
Note that positivity of $\zeta$ is a consequence of assumption
\eqref{eqn2DCoercivity}.

Define
\begin{equation}
\eta_1 \defeq \frac{\zeta^2}{(1+4\sqrt{2})^2L_4^2}>0. \label{def of
eta 1}
\end{equation}

Multiplying \eqref{p equ} with $p$, \eqref{q equ} with $q$, and
adding gives
\begin{align}
\frac12\frac{\partial{h}^2}{\partial{t}}
=\frac{\zeta}{2}&\Delta{h^2}-\zeta\big(|\nabla{p}|^2+|\nabla{q}|^2
\big)-L_4(p\partial_2\partial_2{h}^2-p\partial_1\partial_1{h}^2-2q\partial_1\partial_2h^2)
-ah^2-2ch^4
\non\\
&+L_4p\big[(\partial_2p)^2-(\partial_1p)^2-3(\partial_1q)^2+3(\partial_2q)^2+2\partial_1p\partial_2q+2\partial_2p\partial_1q
\big] \non\\
&+2L_4q\big[\partial_1p\partial_1q-3\partial_1p\partial_2p-\partial_2p\partial_2q-\partial_1q\partial_2q
\big], \label{h equation}
 \end{align}
 where
\begin{equation}\label{def:h} h(x_1, x_2, t) \defeq \sqrt{p^2+q^2}. \end{equation}
Multiplying \eqref{h equation} by $(h^2-\eta_1)^{+}$ and integrating
gives
\begin{align}
 \frac14\frac{d}{dt}\int_{\Omega}|(h^2&-\eta_1)^+|^2dx \non\\
  =-\frac{\zeta}{2}&\int_{\Omega}|\nabla(h^2-\eta_1)^+|^2dx-\zeta\int_{\Omega}(h^2-\eta_1)^+\big(|\nabla{p}|^2
+|\nabla{q}|^2\big)dx+L_4\int_{\Omega}p|\partial_2(h^2-\eta_1)^+|^2dx
\non\\
&+L_4\int_{\Omega}\partial_2{p}\partial_2(h^2-\eta_1)^+ (h^2-\eta_1)^+
dx-L_4\int_{\Omega}p|\partial_1(h^2-\eta_1)^+|^2dx
\non\\
&-L_4\int_{\Omega}\partial_1{p}\partial_1(h^2-\eta_1)^+
(h^2-\eta_1)^+dx-2L_4\int_{\Omega}q\partial_2(h^2-\eta_1)^+\partial_1(h^2-\eta_1)^+dx
\non\\
&-2L_4\int_{\Omega}\partial_1q\partial_2(h^2-\eta_1)^+(h^2-\eta_1)^+dx-\int_{\Omega}2ch^2\Big(\frac{a}{2c}+h^2\Big)(h^2-\eta_1)^+dx
\non\\
&+L_4\int_{\Omega}p\big[(\partial_2p)^2-(\partial_1p)^2-3(\partial_1q)^2+3(\partial_2q)^2+2\partial_1p\partial_2q+2\partial_2p\partial_1q
\big](h^2-\eta_1)^+dx \non\\
&+2L_4\int_{\Omega}q\big[\partial_1p\partial_1q-3\partial_1p\partial_2p-\partial_2p\partial_2q-\partial_1q\partial_2q\big](h^2-\eta_1)^+dx
\non\\
=
-\frac{\zeta}{2}&\int_{\Omega}|\nabla(h^2-\eta_1)^+|^2dx-\zeta\int_{\Omega}(h^2-\eta_1)^+\big(|\nabla{p}|^2
+|\nabla{q}|^2\big)dx+I_1+\cdots+I_9. \label{deriv:maximumbound}
\end{align}
Above we used \eqref{eqnPQICBC},
\eqref{boundary:assumption} and integration by parts.

We estimate the terms $I_1$ through $I_9$ individually. Using the
Schwarz inequality and the fact $|p|+|q| \leq \sqrt{2p^2+2q^2}$, we
obtain
\begin{equation}
 I_1+I_3+I_5 \leq|L_4|\int_{\Omega}(|p|+|q|)|\nabla(h^2-\eta_1)^+|^2dx\leq\sqrt{2}|L_4|\int_{\Omega}h|\nabla(h^2-\eta_1)^+|^2dx.
\end{equation}
Also,
\begin{align}
I_2+I_4 &\leq
|L_4|\int_{\Omega}|\nabla{p}||\nabla(h^2-\eta_1)^+|(h^2-\eta_1)^+dx
\non\\&\leq
\frac{|L_4|}{2}\int_{\Omega}|(h^2-\eta_1)^+|^{\frac32}|\nabla{p}|^2dx
+\frac{|L_4|}{2}\int_{\Omega}|(h^2-\eta_1)^+|^{\frac12}|\nabla(h^2-\eta_1)^+|^2dx
\non\\
&\leq \frac{|L_4|}{2}\int_{\Omega}h(h^2-\eta_1)^+|\nabla{p}|^2dx
+\frac{|L_4|}{2}\int_{\Omega}h|\nabla(h^2-\eta_1)^+|^2dx.
\end{align}
Similarly,
\begin{equation}
I_6\leq|L_4|\int_{\Omega}h(h^2-\eta_1)^+|\nabla{q}|^2dx+|L_4|\int_{\Omega}h|\nabla(h^2-\eta_1)^+|^2dx.
\end{equation}
Furthermore, assumption~\eqref{small coefficient assumption} implies
\begin{equation}
I_7 \leq \int_{\Omega}2ch^2\Big(\frac{|a|}{2c}-h^2
\Big)(h^2-\eta_1)^{+}dx \leq 0. \end{equation} Finally, using the
Cauchy-Schwarz inequality and $|p|+|q|\leq \sqrt{2p^2+2q^2}$ again,
we get \begin{equation} I_8+I_9 \leq
4\sqrt{2}|L_4|\int_{\Omega}h(|\nabla{p}|^2+|\nabla{q}|^2)(h^2-\eta_1)^+dx.
\end{equation}
Combining the above we get
\begin{align}
 \frac14\frac{d}{dt}\int_{\Omega}|(h^2-&\eta_1)^+|^2dx
\non\\
\leq
\frac12\int_{\Omega}&\left[\big(3+2\sqrt{2}\big)|L_4|h-\zeta\right]|\nabla(h^2-\eta_1)^+|^2dx
\non\\
&+\int_{\Omega}\left[\big(1+4\sqrt{2}\big)|L_4|h-\zeta\right](|\nabla{p}|^2+|\nabla{q}|^2)(h^2-\eta_1)^+dx.
\label{equation for h square}
\end{align}
Note that $3+2\sqrt{2} < 1+4\sqrt{2}$, hence if we assume at initial
time
\begin{equation}
|Q_0|=\sqrt{2(p_0^2+q_0^2)}=\sqrt{2}h_0<\frac{\sqrt{2}\zeta}{(1+4\sqrt{2})|L_4|},
\end{equation}
then it follows from \eqref{equation for h square} that
$$
\frac{d}{dt}\int_{\Omega}|(h^2-\eta_1)^+|^2dx \leq 0, \ \ \forall t
\in (0, T),
$$
 which concludes the proof.
\end{proof}

\begin{remark}
For $L_4=0$ the previous result is to be expected, as the energy is
just the usual Dirichlet type energy, up to a null-Lagrangian (see
\eqref{2d coercive part} in the Appendix). The unexpected aspect
captured by the Lemma is that through the gradient type evolution,
the coercive part of the energy manages to control the size of the
badly behaved cubic term that is present for $L_4\not=0$.
\end{remark}

\subsection{Apriori Estimates for Higher Norms.}

For small data, Proposition~\ref{first proposition on 2D small data}
shows that the $L^\infty$ smallness is preserved. Consequently, this
will imply coercivity of the second order terms and positivity of
the energy $\mathcal E$. The main result of this section uses this
and the dissipative energy law~\eqref{eqnEE} to apriori control
higher order norms of $Q$.
\begin{proposition}\label{proposition on 2D lower bound}
For $\Omega\subset\RR^2$ smooth and bounded, there exists an
$\eta_2>0$ depending on $L_i,i=1,2,3,4$ and $\Omega$ so that if:
\begin{align}
&  Q_0 \in H^1(\Omega)\cap L^\infty(\Omega),\, \tilde{Q}\in H^\frac32(\partial\Omega),\non\\
& \|Q_0\|_{L^\infty(\Omega)} \leq \sqrt{2\eta_2},
 \quad\text{and}\quad
 \frac{|a|}{2c} \leq \eta_2,
\end{align}
 then under the coercivity condition \eqref{eqn2DCoercivity} and structural assumptions \eqref{AssumptionsBulk},
 for any $T>0$, and any smooth solution $Q$ of \eqref{eqnQExpansion}--\eqref{eqnQICBC} we have
$$
  \norm{Q}_{L^\infty(0, T; H^1(\Omega))} \leq C
  \quad\text{and}\quad
  \norm{Q}_{L^2(0, T; H^2(\Omega))} \leq C,
$$
for some constant $C$ depending on $T$, $\eta_2$, $\norm{Q_0}_{H^{1}}$ and $\norm{\tilde Q}_{H^{3/2}(\del\Omega)}$.
\end{proposition}
\begin{proof}
As mentioned earlier, the assumption~\eqref{eqn2DCoercivity} guarantees coercivity of the linear terms in 2D and quantitatively gives
\begin{equation}\label{eqn2DCoercivity1}
  \big(L_1|\nabla{Q}|^2+L_2\partial_jQ_{ik}\partial_kQ_{ij}
  +L_3\partial_jQ_{ij}\partial_kQ_{ik}\big)(x) \geq \nu|\nabla{Q}|^2(x),
\end{equation}
where
\begin{equation*}
  \nu\defeq \min\{L_1+L_2, L_1+L_3\}>0.
\end{equation*}
For continuity we prove this in Lemma~\ref{lmaPositiveEnergy} in Appendix~\ref{sxnCoercivity} below, and refer
 the reader to~\cite{LMT87,TE98} for the three dimensional analog.

Now define
\begin{equation}
\eta_1 \defeq \frac{\zeta^2}{(1+4\sqrt{2})^2L_4^2},
\qquad\text{and}\qquad
\eta_2 \defeq \frac{1}{60}\min\left\{\frac{\nu^2}{8L_4^2},
\frac{\zeta^2}{144L_4^2C_1^2},
\eta_1 \right\}, \label{def of eta 2}
\end{equation}
where $C_1$ is the constant appearing in Lemma~\ref{def of C1 and C2}.

To begin, an argument analogous to the proof of Proposition \ref{first proposition on 2D small data} gives
$$
  \|Q(t)\|_{L^\infty(\Omega)} \leq \sqrt{2\eta_2}, \qquad \forall t
  \in [0, T].
$$

Next, we infer from the basic energy law \eqref{eqnEE}, Lemma
\ref{lmaPositiveEnergy},
 that there exists $ \tilde{\eta}=\nu-2|L_4|\sqrt{2\eta_2}>0$, such that
\begin{align*}
\mathcal{E}(Q_0) &\geq \mathcal{E}(Q(t)) \geq
\int_{\Omega}\nu|\nabla{Q}|^2+L_4Q_{lk}\partial_lQ_{ij}\partial_kQ_{ij}
+\frac{a}{2}\tr(Q^2)+\frac{c}{4}\tr^2(Q^2) \,dx \\
&\geq \int_{\Omega}\nu|\nabla{Q}|^2\,dx-|L_4|\|Q\|_{L^\infty(\Omega)}\|\nabla{Q}\|^2
+\frac{c}{4}\int_{\Omega}\left\{\big[\tr(Q^2)+\frac{a}{c}\big]^2-\frac{a^2}{c^2}\right\}\,dx    \\
&\geq \tilde{\eta}\|\nabla{Q}(t)\|^2-\frac{a^2}{4c}|\Omega|.
\end{align*}
Hence $Q \in L^\infty(0, T; H^1(\Omega))$. Furthermore, it follows
from the basic energy law \eqref{eqnEE} and equation
\eqref{eqnQExpansion} that
$$
  Q_t \in L^2(0, T;L^2(\Omega)).
$$
By Lemma \ref{def of C1 and C2},  Proposition~\ref{first proposition
on 2D small data} and Cauchy-Schwarz inequality, we deduce from
\eqref{p equ} and \eqref{q equ} that
\begin{align}
\zeta\|\Delta{p}&(t)\| \non\\
\leq &\|p_t\|+|L_4|\|(\partial_1p)^2-(\partial_1q)^2-(\partial_2p)^2
+(\partial_2q)^2+2\partial_1p\partial_2q+2\partial_2p\partial_1q\| \non\\
&\quad+2|L_4|\|p\partial_1\partial_1p+2q\partial_1\partial_2p-p\partial_2\partial_2p\|+\|ap+2c(p^2+q^2)p\|
\non\\
\leq&\|p_t\|+2|L_4|\|(\partial_1p)^2+(\partial_2p)^2+(\partial_1q)^2+(\partial_2q)^2\|
+2|L_4|\|p\|_{L^\infty(\Omega)}\|\partial_1\partial_1p\| \non\\
&\quad+4|L_4|\|q\|_{L^\infty(\Omega)}\|\partial_1\partial_2p\|
+2|L_4|\|p\|_{L^\infty(\Omega)}\|\partial_2\partial_2p\|+C \non\\
\leq&\|p_t\|+2|L_4|C_1\Big[\|p\|_{L^\infty}\big(\|\Delta{p}\|+\|p\|+\|\tilde{p}\|_{H^\frac32(\partial\Omega)}
\big)
+\|q\|_{L^\infty}\big(\|\Delta{q}\|+\|q\|+\|\tilde{q}\|_{H^\frac32(\partial\Omega)} \big) \Big]\non\\
&\quad+8|L_4|C_1\|h\|_{L^\infty}\big(\|\Delta{p}\|+\|p\|+\|\tilde{p}\|_{H^\frac32(\partial\Omega)} \big)+C \non\\
\leq&\|p_t\|+|L_4|C_1\|h\|_{L^\infty}\big(10\|\Delta{p}\|+2\|\Delta{q}\|
\big)+C. \label{boundary used}
\end{align}

Here $C$ depends on $\Omega$, $Q_0$,
$\tilde{Q}$, and the coefficients of the system. Analogously, we
know
\begin{equation}
\zeta\|\Delta{q}(t)\|\leq\|q_t\|+|L_4|C_1\|h\|_{L^\infty}\big(10\|\Delta{q}\|+2\|\Delta{p}\|\big)+C.
\end{equation}
After summing up, we get
\begin{equation}\label{eq:LaplacDominatedEst}
\zeta\big(\|\Delta{p}(t)\| +\|\Delta{q}(t)\| \big)\leq\|Q_t\|
+12|L_4|C_1\|h\|_{L^\infty}\big(\|\Delta{q}\|+\|\Delta{p}\| \big)+C,
\end{equation}
which yields the bound of $\|\Delta{Q}\|$ in $L^2(0, T)$, due to the
choice of $\eta_2$ and the fact $\|h\|_{L^\infty(0, T; \Omega)}\leq
\sqrt{\eta_2}$.
\end{proof}

\begin{remark}\label{remark:60constant}
The factor $\frac{1}{60}$ in \eqref{def of eta 2} is not used the
proof above. However, it will be necessary in the proof of Theorem
\ref{thm2Dgexist}, part (i), as described in the discussion before
\eqref{est:H2potential2}.
\end{remark}

\subsection{Weak Solutions}

The purpose of this section is to show that the apriori estimates previously established are enough to
show global existence and uniqueness of weak solutions for small initial data.
While this is usually standard, the nonlinearity appearing in the higher order terms makes things complicated in our situation.
Specifically, we crucially need $\norm{Q}_{L^\infty}$ to be small in order to obtain coercivity of the second order terms.
Thus any approximating scheme devised to prove the existence of weak solutions must preserve $L^\infty$ smallness of the initial data.
Since $Q$ is a $2 \times 2$ matrix we don't have the luxury of a maximum principle that apriori preserves $\norm{Q}_{L^\infty}$,
and the approximating scheme must be constructed carefully.
We carry out this construction below.

We begin by recalling the definition of weak solutions in our context.

\begin{definition}\label{def of weak solution}

For any $T\in (0,+\infty)$, a function $Q$ satisfying
$$
   Q \in L^\infty(0, T; H^1\cap L^\infty) \cap L^2(0, T; H^2),\quad
   \partial_t Q \in L^2(0, T; L^2),
   \quad\text{and}\quad
   Q \in S^{(2)}, \quad \text{a.e. in } \ \Omega\times (0, T),
$$
is called a weak solution of the problem \eqref{p equ}-\eqref{q
equ}, if it satisfies the initial and boundary conditions~\eqref{eqnPQICBC}, and we have
\begin{align*}
-\int_{\Omega\times [0,T]} &Q: \partial_t R\,dx \, dt\\
=&-2L_1\int_{\Omega\times [0,T]}\partial_kQ:\partial_kR\,dx \, dt-\int_{\Omega\times [0,T]}\big[a+c\tr(Q^2)\big]Q:R\,dx \, dt\\
&\quad-2(L_2+L_3)\int_{\Omega\times
[0,T]}\partial_kQ_{ik}\partial_jR_{ij}\,dx \, dt
+(L_2+L_3)\int_{\Omega\times
[0,T]}\partial_kQ_{lk}\partial_lR_{ii}\,dx \, dt\\
&\quad-2L_4\int_{\Omega\times [0,T]}
Q_{lk}\partial_kQ_{ij}\partial_lR_{ij}\,dx \, dt
-L_4\int_{\Omega\times [0,T]}\partial_iQ_{kl}\partial_jQ_{kl}R_{ij}\,dx \, dt \\
&\quad+\frac{L_4}{2}\int_{\Omega\times [0,T]}|\nabla{Q}|^2R_{ii}\,dx \, dt-\int_\Omega Q_0:R(0)\,dx.
\end{align*}
Here $R\in C_c^\infty([0,T)\times \Omega,M^{2\times 2}(\mathbb{R}))$ is arbitrary.%
\end{definition}
\begin{remark}
The notion of weak solution above is similar to the one considered
in \cite[Definition~3.2, Remark~4]{SL09} for a related system. The
main difference in our situation is the regularity requirement on
$R$. The more standard requirement would be that $R\in
H^1_0(\Omega)$, however, because of the presence of the cubic term
we need a stronger assumption. A similar situation occurs in
\cite{CR13}, where test functions are taken in a smaller space than
$H^1_0(\Omega)$ because of the presence of similar terms. For
simplicity we take here only smooth functions although a larger
class of test function should still work for obtaining existence of
solutions.
\end{remark}

\begin{proof}[Proof  of Theorem~\ref{thm2Dgexist}, part {\it (i)}]
For simplicity we only consider the homogeneous boundary problem.
The analysis of the corresponding inhomogeneous boundary condition is similar but involves many lengthy computations which obscure the heart of the matter.
We start by augmenting \eqref{p equ}-\eqref{q equ}, \eqref{eqnQExpansion} with a singular potential.
Explicitly, consider the system
\begin{align}
\frac{\partial{p}}{\partial{t}}&=\zeta\Delta{p}-\eps \frac{\partial
f}{\partial
p}(p,q)+L_4\big[(\partial_1p)^2-(\partial_1q)^2-(\partial_2p)^2
+(\partial_2q)^2+2\partial_1p\partial_2q+2\partial_2p\partial_1q \big] \non\\
&\qquad+2L_4(p\partial_1\partial_1p+2q\partial_1\partial_2p-p\partial_2\partial_2p)-ap-2c(p^2+q^2)p,
\label{p equ+} \\
\frac{\partial{q}}{\partial{t}}&=\zeta\Delta{q}-\eps\frac{\partial
f}{\partial
q}(p,q)+2L_4\big[\partial_1q\partial_2q-\partial_1p\partial_2p+\partial_1p\partial_1q
-\partial_2p\partial_2q \big] \non\\
&\qquad+2L_4(p\partial_1\partial_1q+2q\partial_1\partial_2q-p\partial_2\partial_2q)-aq-2c(p^2+q^2)q,
\label{q equ+}
\end{align}
with initial data
\begin{equation}
p(x, 0) =p_0^\eps(x), \qquad q(x)=q_0^\eps(x), \quad\forall
x\in\Omega\label{IC+}
\end{equation}
and boundary conditions
\begin{equation}
p(x, t) =0, \qquad q(x,t)=0, \quad\forall x\in\partial\Omega
\label{BC+}.
\end{equation}
Here $p_0^\eps, q_0^\eps\in C^\infty(\Omega)\cap H^1_0(\Omega)$ are such that
$$
 p_0^\eps\to p_0
 \quad\text{and}\quad
 q_0^\eps\to q_0
 \quad\text{in } H^1(\Omega)\cap L^\infty(\Omega),
$$
and $f(p,q)$ is the singular potential%
\footnote{ Let us note that this choice of the singular
potential ensures that the system thus obtained satisfies the
symmetry and tracelesness constraints. The partial derivatives will
only make sense for solutions of finite energy, hence  such that
$p^2+q^2<4\eta_2$ a.e. so that we are in the effective domain of the
convex potential $f$}
$$
  f(p,q)=\left\{\begin{array}{ll}
     -\ln(4\eta_2-p^2-q^2) &\textrm{ if }p^2+q^2<4\eta_2\\
   \infty &\textrm{ if }p^2+q^2\ge 4\eta_2.\end{array}\right.
$$ (where $\eta_2$ is as defined in \eqref{def of eta 2}).
The advantage of this approximating system is that it has a singular
potential term in its energy:
 \begin{align}\label{def:energyEeps}
\mathcal{E}^\eps[p,q]&\defeq \int_\Omega \zeta\big(|\nabla
p|^2+|\nabla q|^2\big)+\eps f(p,q)\,dx
+2L_4\int_\Omega{p}\big(|\partial_1p|^2-|\partial_2p|^2\big)\,dx\non\\
&\qquad
+4L_4\int_\Omega{q}(\partial_1p\partial_2p+\partial_1q\partial_2q)\,dx+\int_\Omega
a(p^2+q^2)+c(p^2+q^2)^2\,dx.
\end{align}
Hence finite energy will imply $p^2+q^2\le 4\eta_2$ almost
everywhere. We will then prove an additional preservation of
smallness principle for the approximate system \eqref{p
equ+}-\eqref{BC+}. Namely, we will show the stronger $L^\infty$
bound $p^2+q^2\le \eta_2$ almost everywhere in time and space, provided this is true
initially, at $t=0$. Thus we will be able to conclude that  the terms coming
from the singular potential become uniformly small and disappear in
the limit $\eps\to 0$.

In order to obtain the existence of the approximate system \eqref{p
equ+}-\eqref{BC+}, we regularize the singular potential $f$ and
construct an approximating sequence using the Galerkin method. To
regularize the singular potential we use an  approximating sequence
of functions $f_N:\RR^2\to \RR$ that satisfy the following
properties:
\begin{enumerate}
\item $f_N:\RR^2\to \RR$ is $C^\infty$ and convex,
\item There exists a constant $\alpha\in\RR$ such that
 \begin{equation}\label{properties:fapprox}
   -\alpha^2\le f_N(p,q), \forall p,q\in \RR\textrm{ and }\forall N\ge
1,
  \end{equation}
\item $f_N\le f_{N+1}\le f$ on $\RR^2$ for all $N\in\NN$,
\item $f_N\to f$ in $C^k(D(f))$ as $N\to\infty$ (where $D(f)$ is the domain of $f$, namely $D(f):=\{(p,q)\in\RR^2; p^2+q^2<4\eta_2\}$).
\end{enumerate}
A similar construction was carried out in~\cite{MW12} using Moreau-Yosida approximation and a suitable smoothing,
and we refer the reader to~\cite{MW12} for the details.

For the Galerkin approximation, let $\{\varphi_1,\dots,\varphi_n,\dots\}$ be an orthonormal basis of $L^2(\Omega)$ consisting of eigenvectors of the
Laplacian (with zero Dirichlet boundary conditions).
Let
$\mathcal{P}_m:L^2\to H_m$ where
$H_m\defeq\textrm{span}\{\varphi_1,\dots,\varphi_m\}$.
Consider  the finite dimensional system
\begin{align}
\frac{\partial{p_m}}{\partial{t}}=&\zeta\Delta{p_m}-\eps
\mathcal{P}_m\Big\{\frac{\partial f_N}{\partial p}(p_m,q_m)\Big\}
+L_4\mathcal{P}_m\big\{(\partial_1p_m)^2-(\partial_1q_m)^2-(\partial_2p_m)^2\big\}\non\\
&+L_4\mathcal{P}_m\big\{(\partial_2q_m)^2+2\partial_1p_m\partial_2q_m+2\partial_2p_m\partial_1q_m\big\}\non\\
&+2L_4\mathcal{P}_m\big\{p_m\partial_1\partial_1p_m+2q_m\partial_1\partial_2p_m-p_m\partial_2\partial_2p_m-ap_m-2c(p_m^2+q_m^2)p_m\big\}
\label{p equ++} \\
\frac{\partial{q_m}}{\partial{t}}=&\zeta\Delta{q_m}-\eps
\mathcal{P}_m\Big\{\frac{\partial f_N}{\partial q}(p_m,q_m)\Big\}
+2L_4\mathcal{P}_m\big\{\partial_1q_m\partial_2q_m-\partial_1p_m\partial_2p_m\big\}\non\\
&+2L_4\mathcal{P}_m\big\{\partial_1p_m\partial_1q_m-\partial_2p_m\partial_2q_m\big\} \non\\
&+2L_4\mathcal{P}_m\big\{p_m\partial_1\partial_1q_m+2q_m\partial_1\partial_2q_m-p_m\partial_2\partial_2q_m-aq_m-2c(p_m^2+q_m^2)q_m\big\}
\label{q equ++}
\end{align}
with initial conditions
\begin{equation}
p_m(x, 0) =(\mathcal{P}_m p_0^\eps)(x),
\qquad
q(x)=(\mathcal{P}_mq_0^\eps)(x), \quad
\forall x\in\Omega \label{IC++}
\end{equation}

The above system depends on three parameters: $\eps$, $m$ and $N$.
For simplicity we drop the explicit dependence on $\eps$ and $N$
from the notation, and only keep the subscript $m$ in the solutions
$p_m,q_m$. We will first send $N\to\infty$ and then $m\to\infty$ to
obtain solutions to the approximate continuous system  \eqref{p
equ+}--\eqref{q equ+}. Finally we will pass to the limit $\eps\to
0$. We divide the remainder of the proof into three steps.

\bigskip
\step{Step 1: Sending $N\rightarrow\infty$.}
We look for solutions of the form
$$
p_m(t,x)\defeq \sum_{i=1}^m a^i_m(t)\varphi_i(x),\quad
q_m(t,x)=\sum_{i=1}^m b^i_m(t)\varphi_i(x)
$$
The existence of solutions for short time is a consequence of the
standard Cauchy-Peano local existence theory for systems of ordinary
differential equations. The bounds \eqref{est:iteratedinftyN}
obtained below will suffice for showing that the existence of the
system holds for arbitrary intervals of time.

Note that for $\eps>0$ small enough we have
$(p_0^\eps)^2+(q_0^\eps)^2<2\eta_2$ almost everywhere.
Since for $m\to\infty$ we have  $\mathcal{P}_mp_0^\eps\to p_0^\eps$ in $H^2\hookrightarrow L^\infty$ we can arrange
\begin{equation}\label{eq:compatibleInitialData}
\|\mathcal{P}_m{p}_0^\eps\|_{L^\infty}^2+\|\mathcal{P}_mq_0^\eps\|_{L^\infty}^2<2\eta_2,
\end{equation}
for $m = m(\eps)$ large enough, and $\eps$ sufficiently small.

Multiplying equation \eqref{p equ++} by $\partial_tp_m$, and
equation \eqref{q equ++} by $\partial_tq_m$, adding and integrating over $\Omega$ gives
\begin{multline}\label{est:energyineqdisc}
\mathcal{E}[p_m(t), q_m(t)]+\eps\int_\Omega f_N(p_m,q_m)\,dx+\|\partial_t p_m\|_{L^2((0,T)\times\Omega)}+\|\partial_t q_m\|_{L^2((0,T)\times\Omega)}\\
\le \mathcal{E}^\eps [p_m(0),q_m(0)].
\end{multline}
Here $\mathcal{E}[v,w]\defeq \mathcal{E}^\eps[v,w]-\eps \int_\Omega
f_N(v,w)\,dx$ with $\mathcal{E}^\eps$ as defined in
\eqref{def:energyEeps}.
To obtain~\eqref{est:energyineqdisc} we integrated by parts and used the fact that $\mathcal P_m$ is a self-adjoint operator on $L^2$.

We now focus on understanding what a priori bounds are provided
by \eqref{est:energyineqdisc}. We claim that the finite dimensionality of $H_m$ allows us to find a large enough
constant $C(m)$, which depends on $m$ but not on $N$ or $\eps$, such
that
\begin{equation}\label{est:finitedimcoerciv}
\int_\Omega \frac{\zeta}{2}\big(|\nabla p_m(t)|^2+|\nabla
q_m(t)|^2\big)\,dx\le \mathcal{E}[p_m(t),q_m(t)]+C(m).
\end{equation}

To see this, observe that there exists a constant
$\tilde C(m)$ depending only on $m$ and $\Omega$, such that
\begin{align*}
&\int_\Omega 2L_4\Big[ p_m\big(|\partial_1p_m|^2-|\partial_2p_m|^2\big)+2q_m\big(\partial_1p_m\partial_2p_m+\partial_1q_m\partial_2q_m\big)\Big]\,dx\\
&\quad+\int_\Omega a(p_m^2+q_m^2)+c(p_m^2+q_m^2)^2\,dx\\
&\le C\int_\Omega  \frac{2}{3}L_4(p_m^3+q_m^3)+\frac{4}{3}L_4|\nabla p_m|^3+\frac{4}{3}L_4|\nabla q_m|^3+a(p_m^2+q_m^2)+c(p_m^2+q_m^2)^2\,dx\\
&\le C\int_\Omega
L_4(p_m^3+q_m^3)+a(p_m^2+q_m^2)+c(p_m^2+q_m^2)^2\,dx+L_4\tilde
C(m)\int_\Omega \big(p_m^3+q_m^3\big)\,dx.
\end{align*} (where for the first inequality we used Young's inequality $ab\le \frac{a^3}{3}+\frac{2b^{\frac{3}{2}}}{3}$
and for the second the finite dimensionality of $H_m$).
This immediately implies~\eqref{est:finitedimcoerciv} as claimed.

For the rest of this Step, for the sake of clarity we will specify the hidden dependence on $N$, namely denote $p_m^N\defeq p_m$, $q_m^N\defeq q_m$.

Thus using~\eqref{properties:fapprox}, adding $C(m) + \eps \alpha^2 \abs{\Omega}$ to both sides of \eqref{est:energyineqdisc}
and taking into account \eqref{est:finitedimcoerciv},  we have the apriori bounds
\begin{align*}
  \norm{p_m^N}_{L^\infty(0,T;H^1)} &\leq C,
  &
  \norm{\partial_t p_m^N}_{L^2(0,T;L^2)} &\leq C,
  \\
  \norm{q_m^N}_{L^\infty(0,T;H^1)} &\leq C,
  &
  \norm{\partial_t q_m^N}_{L^2(0,T;L^2)} &\leq C,
\end{align*}
where the constant $C$ is independent of $N$ but depending on $m$. Further, since
$p_m^N,q_m^N\in H_m$ and $H_m$ is a finite dimensional space with a
$C^\infty$ basis the above implies
\begin{equation}\label{est:iteratedinftyN}
\sup_{N\in\NN}\|p_m^N\|_{L^\infty(0,T;H^k)}
+\|q_m^N\|_{L^\infty(0,T;H^k)}+\|\partial_t p_m^N\|_{L^2(0,T;H^k)}
+\|\partial_t q_m^N\|_{L^2(0,T;H^k)}<\infty \quad\forall k\in\NN.
\end{equation}

The above estimates show that as $N\to\infty$, the limit of $p_m^N$
and $q_m^N$ exist (along a subsequence) and in suitable spaces, to be denoted $p_m$, respectively $q_m$. Further, using the above
apriori estimates in~\eqref{est:energyineqdisc} we obtain
\begin{equation*}
  \eps\int_\Omega f_N(p_m^N,q_m^N)\,dx\leq C,
\end{equation*}
where the constant $C$ is independent of $N$. In particular, using
the monotonicity of $f_N(\cdot,\cdot)$ with respect to $N$, we have
that for any $N_0\ge 1$:
\begin{equation*}
\eps\int_\Omega f_{N_0}(p_m^N,q_m^N)\,dx\leq C,\forall N\ge N_0
\end{equation*} hence using the pointwise convergence of $p^N_m,q^N_m$ respectively to $p_m,q_m$ we get:
\begin{equation*}
  \eps\int_\Omega f_{N_0}(p_m,q_m)\,dx\leq C,
\end{equation*}

Since $N_0$ was chosen arbitrarily the monotone convergence theorem
now implies
\begin{equation}\label{bound of f indep of N}
\eps\int_\Omega f(p_m,q_m)\,dx\leq C,
\end{equation}
in the limit $N \to \infty$, along a subsequence. Thus, as $N \to
\infty$ along a subsequence,  we obtain a solution to \eqref{p
equ++}--\eqref{q equ++} with $f_N$ replaced by $f$. Further
\eqref{bound of f indep of N} shows that for all $t > 0$ the
limiting functions $p_m$ and $q_m$ are in the effective domain of
the convex potential $f$.

\bigskip
\step{Step 2: Sending $m\to\infty$.}
Since~\eqref{bound of f indep of N} implies $p_m^2+q_m^2< 4\eta_2$ for all $m \in \NN$, almost everywhere
 in $(0, T)\times\Omega$, the same argument as in Proposition~\ref{proposition on 2D lower bound} now shows
$$
 2 \tilde{\eta}\int_{\Omega}|\nabla{p}_m(x,t)|^2+|\nabla{q}_m(x,t)|^2dx-\frac{a^2}{4c}|\Omega|
  \leq \mathcal{E}[p_m(t),q_m(t)], \qquad \forall t>0, \, m\in
  \mathbb{N}.
$$
Using~\eqref{est:energyineqdisc} (with $f_N$ replaced by $f$) shows the existence of a constant $C(\eps)$ such that
\begin{equation}\label{uniform bound indep of m}
\int_0^T\int_\Omega\big(|\partial_tp_m|^2+|\partial_tq_m|^2\big)\,dx \, dt
< C(\eps).
\end{equation}
Since we work on a domain where $f$ is finite almost everywhere,
equations~\eqref{p equ++} and~\eqref{q equ++} (with $f_N$ replaced
by $f$) show
\begin{align}\label{uniform bound dep on eps}
&\int_0^T\int_\Omega \left(\zeta\Delta{p_m}-\eps\mathcal{P}_m\left\{\frac{2 p_m}{4\eta_2-(p_m^2+q_m^2)}\right\}+\mathcal{G}_m\right)^2\,dx\,dt\non\\
&+\int_0^T\int_\Omega
\left(\zeta\Delta{q_m}-\eps\mathcal{P}_m\left\{\frac{2
q_m}{4\eta_2-(p_m^2+q_m^2)}\right\}+\mathcal{H}_m\right)^2\,dx\,dt<C(\eps).
\end{align}
The quantities $\mathcal G_m$ and $\mathcal H_m$ above are defined by
\begin{gather}
  \label{eqnG}
  \mathcal{G}_m\defeq \frac{\partial{p_m}}{\partial{t}}-\zeta\Delta{p_m}+\eps\mathcal{P}_m\left\{\frac{2 p_m}{4\eta_2-(p_m^2+q_m^2)}\right\}
  \\
  \label{eqnH}
  \llap{\text{and}\qquad}
  \mathcal{H}_m\defeq \frac{\partial{q_m}}{\partial{t}}-\zeta\Delta{q_m}+\eps\mathcal{P}_m\left\{\frac{2 q_m}{4\eta_2-(p_m^2+q_m^2)}\right\}.
\end{gather}
Expanding the L.H.S. of \eqref{uniform bound indep of m}, we have
\begin{align}
C(\eps) &>\int_0^T\int_{\Omega}\big(|\zeta\Delta{p_m}|^2+|\zeta\Delta{q_m}|^2\big)\,dx \, dt
+\int_0^T\int_{\Omega}\big(|\mathcal{G}_m|^2+|\mathcal{H}_m|^2\big)\,dx \, dt\non\\
&\qquad+\int_0^T\int_{\Omega}\left|\eps\mathcal{P}_m\Big\{\frac{2p_m}{4\eta_2-(p_m^2+q_m^2)}\Big\}\right|^2
 +\left|\eps\mathcal{P}_m\Big\{\frac{2q_m}{4\eta_2-(p_m^2+q_m^2)}\Big\}\right|^2dxdt\non\\
&\qquad-\int_0^T\int_{\Omega}2\zeta\eps\Delta{p}_m\mathcal{P}_m\Big\{\frac{2p_m}{4\eta_2-(p_m^2+q_m^2)}\Big\}
 +2\zeta\eps\Delta{q}_m\mathcal{P}_m\Big\{\frac{2q_m}{4\eta_2-(p_m^2+q_m^2)}\Big\}\,dx \, dt\non\\
&\qquad-\int_0^T\int_{\Omega}2\eps\mathcal{P}_m\Big\{\frac{2p_m}{4\eta_2-(p_m^2+q_m^2)}\Big\}\mathcal{G}_m
 +2\eps\mathcal{P}_m\Big\{\frac{2q_m}{4\eta_2-(p_m^2+q_m^2)}\Big\}\mathcal{H}_m\,dx \, dt\non\\
&\qquad+\int_0^T\int_{\Omega}2\zeta\big(\Delta{p}_m\mathcal{G}_m+\Delta{q}_m\mathcal{H}_m\big)\,dx \, dt\non\\
&\defeq I_1+\cdots +I_6.
\end{align}
Clearly $I_1, I_2$ and $I_3$ are positive.
For $I_4$, we integrate by parts to obtain
\begin{align*}
I_4&=4\zeta\eps\int_0^T\int_{\Omega}\nabla{p}_m\cdot\nabla\Big\{\frac{2p_m}{4\eta_2-(p_m^2+q_m^2)}\Big\}
+\nabla{q}_m\cdot\nabla\Big\{\frac{2q_m}{4\eta_2-(p_m^2+q_m^2)}\Big\}\,dx \, dt\\
&=4\zeta\eps\int_0^T\int_{\Omega}\frac{4\eta_2\big(|\nabla{p}_m|^2+|\nabla{q}_m|^2\big)}{4\eta_2-(p_m^2+q_m^2)}\,dx \, dt\\
&\qquad+4\zeta\eps\int_0^T\int_{\Omega}\frac{p_m^2|\nabla{p}_m|^2+q_m^2|\nabla{q}_m|^2
-p_m^2|\nabla{q}_m|^2-q_m^2|\nabla{p}_m|^2+4p_mq_m\nabla{p}_m\nabla{q}_m}{4\eta_2-(p_m^2+q_m^2)}\,dx \, dt\\
&\geq 4\zeta\eps\int_0^T\int_{\Omega}\frac{2p_m^2|\nabla{p}_m|^2+2q_m^2|\nabla{q}_m|^2+4p_mq_m\nabla{p}_m\nabla{q}_m}{4\eta_2-(p_m^2+q_m^2)}\,dx \, dt\\
&\geq 0.
\end{align*}
Here we used the fact that $p_m^2+q_m^2< 4\eta_2$ a.e. in $(0, T)\times\Omega$.

By Young's inequality we see
\begin{align*}
I_5+I_6 &\geq -\frac12(I_1+I_3)-16I_2.
\end{align*}
Consequently
$$
  C(\eps) > \frac12(I_1+I_3)-15 I_2.
$$
We claim that due to our choice of $\eta_2$, the $15 I_2$ term can
be hidden in $I_1 / 2$. Indeed, using~\eqref{p equ++}, \eqref{q
equ++}, \eqref{eqnG} and \eqref{eqnH} we see that $\mathcal G_m$ and
$\mathcal H_m$ are respectively all the terms in~\eqref{p equ++}
and~\eqref{q equ++} that have $L_4$ as a coefficient. Of these, the
second order terms are all multiplied by $p_m$ or $q_m$, both are
which are uniformly bounded by $\eta_2$. The first order terms can
be handled by interpolation. Consequently when $\eta_2$ is
sufficiently small we can arrange $\abs{15 I_2} \leq I_1 / 4$ (see also Remark~\ref{remark:60constant}).

The above shows
\begin{equation}\label{est:H2potential2} \int_0^T\int_\Omega \delta_0 |\Delta
p_m|^2+\delta_0 |\Delta q_m|^2\,dx\,dt\le C(\eps),
\end{equation}
for some small constant $\delta_0>0$ independent of $m$.
This allows us to pass to the limit $m\to\infty$ and obtain weak solutions
of \eqref{p equ+},\eqref{q equ+}.
Moreover, these solutions are such
that the limits $p^\eps,q^\eps$ belong to $L^\infty(0,T;H^1\cap L^\infty)\cap
L^2(0,T;H^2)$.
Since $H^1\hookrightarrow
L^6$ we now have $p^\eps,q^\eps\in L^4(0,T;W^{1,3})$.
Consequently using the definition of weak solutions we see that \eqref{p equ+},\eqref{q equ+} hold pointwise with all the terms interpreted
as elements of $L^2(0,T;L^{3/2})$.

\step{Step 3: Sending $\eps\to0$.}
We recall that for clarity of presentation we have suppressed the $\eps$ superscript, and $p, q$ are
solutions of the $\eps$ dependent system \eqref{p equ+}--\eqref{q equ+}.
Since all terms in the equation \eqref{p equ+},\eqref{q equ+} are $L^2(0,T;L^{3/2})$ we can use the
same argument we used in the proof of Proposition~\ref{first proposition on 2D small data}.
Namely letting $h^2 = p^2 + q^2$, multiplying \eqref{p equ+} by $p(h^2-\eta_2)^+$, \eqref{q equ+}
by $q(h^2-\eta_2)^+$, adding and integrating over $\Omega$ leads to the analogue of~\eqref{equation for h square}:
\begin{align}
 \frac14\del_t \int_{\Omega}|(h^2-\eta_2)^+|^2(t) dx
 &\leq
  \frac12\int_{\Omega} \left[\big(3+2\sqrt{2}\big)|L_4|h-\zeta\right]|\nabla(h^2-\eta_2)^+|^2dx
\non\\
&\qquad+
\int_{\Omega}\left[\big(1+4\sqrt{2}\big)|L_4|h-\zeta\right](|\nabla{p}|^2+|\nabla{q}|^2)(h^2-\eta_2)^+dx
\non\\
&\qquad+\int_{\Omega}
-\eps\frac{2h^2(h^2-\eta_2)^+}{4\eta_2-(p^2+q^2)}\,dx.
\label{equation for h square+}
\end{align}

Recall that we chose the initial data such that for the $\eps>0$ small enough we have
$$
  \|h(0,\cdot)\|_{L^\infty}^2
  <\eta_2
  <\eta_1
  =\frac{\zeta^2}{(1+4\sqrt{2})^2L_4^2}.
$$
Inequality~\eqref{equation for h square+} shows that
$$
 \del_t \norm{ (h(t)^2-\eta_2)^+ }_{L^2}^2 \leq 0
 \qquad\text{provided}\qquad \norm{h(t)}_{L^\infty}^2 \leq \eta_1.
$$
This immediately shows that if $\norm{(h(t)^2 - \eta_2)^+}_{L^2}^2 = 0$ at time $0$, it must remain $0$ for all $t \geq 0$.
Consequently $p^2+q^2<\eta_2$ for all $t \geq 0$.

This immediately shows that $\abs{p \del_p f(p, q)} \leq C(\eta_2)$, and the extra
 $\eps$-terms appearing in~\eqref{p equ+}--\eqref{p equ+} converge to $0$ uniformly as $\eps \to 0$.
Following the proof of Proposition~\ref{proposition on 2D lower bound} this will now give~\eqref{boundary used}
 with additional $\eps$ terms that are uniformly converging to $0$.
This gives uniform in $\eps$ estimates for $p,q$ in $L^2(0,T;H^2)$ and for $\del_t p, \del_t q$ in $L^2( 0, T; L^2 )$,
 which is enough to pass to the limit $\eps\to 0$.
\end{proof}
\begin{lemma}\label{lemma on continuous dependence}
Suppose
$$
  Q_i=\left(
  \begin{array}{cc}
   p_i & {q}_i \\
   {q}_i & -{p}_i \\
   \end{array}
   \right) \in L^\infty(0, \infty; H^1(\Omega))\cap L_{loc}^2(0,
   \infty; H^2(\Omega)) \ \ (i=1,2)
$$
are two global weak solutions to the problem \eqref{p equ}-\eqref{eqnPQICBC} on $(0, T)$, which satisfy
$$
  \|{Q}_i\|_{L^\infty((0,\infty)\times\Omega)}\leq\sqrt{2\eta_2} \ (i=1, 2),$$ with $\eta_2$ as in Theorem~\ref{thm2Dgexist}.

Then for any $t \in (0, T)$, we have
\begin{equation}
\|(Q_1-Q_2)(t)\| \leq Ce^{Ct}\|Q_{01}-Q_{02}\|,
 \label{continuous dependence equation}
\end{equation}
where $C>0$ is a constant that depends on $\Omega$, $Q_{0i}$ $(i=1,
2)$, $\tilde{Q}$ and the coefficients of the system, but not $t$.
\end{lemma}
\begin{proof}
Let $\bar{p}={p}_1-{p}_2$, $\bar{q}={q}_1-{q}_2$. We see
\begin{align}
\bar{p}_t
=\zeta\Delta&\bar{p}-a\bar{p}-2c(p_1^2+p_1p_2+p_2^2+q_1^2)\bar{p}-2cp_2(q_1+q_2)\bar{q} \non \\
&+L_4\big[
\partial_1\bar{p}\partial_1(p_1+p_2)-\partial_1\bar{q}\partial_1(q_1+q_2)-\partial_2\bar{p}\partial_2(p_1+p_2)\big]\non\\
&+L_4\big[\partial_2\bar{q}\partial_2(q_1+q_2)+2\partial_1\bar{p}\partial_2q_1+2\partial_1p_2\partial_2\bar{q}+2\partial_1q_1\partial_2\bar{p}+2\partial_2p_2\partial_1\bar{q}
\big] \non \\
&+2L_4\big(\bar{p}\partial_1\partial_1p_1+p_2\partial_1\partial_1\bar{p}+2\bar{q}\partial_1\partial_2p_1+2q_2\partial_1\partial_2\bar{p}-\bar{p}\partial_2\partial_2p_1
-p_2\partial_2\partial_2\bar{p} \big), \label{p bar equ}
\\
\shortintertext{and}
\bar{q}_t
=\zeta\Delta&\bar{q}-a\bar{q}-2c(q_1^2+q_1q_2+q_2^2+p_2^2)\bar{q}-2cq_1(p_1+p_2)\bar{p} \non\\
&+2L_4\big[\partial_1\bar{q}\partial_2q_1+\partial_1q_2\partial_2\bar{q}-\partial_1\bar{p}\partial_2p_1-\partial_1p_2\partial_2\bar{p}\big]\non\\
&+2L_4\big[\partial_1\bar{p}\partial_1q_1+\partial_1p_2\partial_1\bar{q}-\partial_2\bar{p}\partial_2q_1-\partial_2p_2\partial_2\bar{q} \big]  \non\\
&+2L_4\big(\bar{p}\partial_1\partial_1q_1+p_2\partial_1\partial_1\bar{q}+2\bar{q}\partial_1\partial_2q_1+2q_2\partial_1\partial_2\bar{q}-\bar{p}\partial_2\partial_2q_1
-p_2\partial_2\partial_2\bar{q} \big),
\label{q bar equ}\\
\bar{p}(0,x)=&\bar{q}(0,x)=0, \ \forall x \in \Omega, \ \
\bar{p}|_{\partial\Omega}=\bar{q}|_{\partial\Omega}=0.
 \label{IC for p bar equ}
 \end{align}

Multiplying equation \eqref{p bar equ} with $\bar{p}$, equation
\eqref{q bar equ} with $\bar{q}$, integrating over $\Omega$ and
using the boundary condition \eqref{IC for p bar equ} gives
\begin{align}
\frac12\frac{d}{dt}\big(\|\bar{p}\|^2+\|\bar{q}\|^2&\big)+\zeta\|\nabla\bar{p}\|^2+\zeta\|\nabla\bar{q}\|^2\non\\
=L_4\int\big[
\partial_1&\bar{p}\partial_1(p_1+p_2)-\partial_1\bar{q}\partial_1(q_1+q_2)-\partial_2\bar{p}\partial_2(p_1+p_2)+\partial_2\bar{q}\partial_2(q_1+q_2)\non\\
&+2\partial_1\bar{p}\partial_2q_1+2\partial_1p_2\partial_2\bar{q}+2\partial_1q_1\partial_2\bar{p}+2\partial_2p_2\partial_1\bar{q}
\big]\bar{p}\,dx\non\\
-\int{a}\bar{p}^2&+2c(p_1^2+p_1p_2+p_2^2+q_1^2)\bar{p}^2+2cp_2\big[(q_1+q_2)\bar{q}\big]\bar{p}\,dx \non\\
+2L_4\int&\big(\bar{p}\partial_1\partial_1p_1+p_2\partial_1\partial_1\bar{p}+2\bar{q}\partial_1\partial_2p_1+2q_2\partial_1\partial_2\bar{p}-\bar{p}\partial_2\partial_2p_1
-p_2\partial_2\partial_2\bar{p} \big)\bar{p}\,dx \non\\
+2L_4\int&\big[
\partial_1\bar{q}\partial_2q_1+\partial_1q_2\partial_2\bar{q}-\partial_1\bar{p}\partial_2p_1-\partial_1p_2\partial_2\bar{p}
+\partial_1\bar{p}\partial_1q_1+\partial_1p_2\partial_1\bar{q}\non\\
&-\partial_2\bar{p}\partial_2q_1-\partial_2p_2\partial_2\bar{q}
\big]\bar{q}\,dx\non\\
-\int{a}\bar{q}^2&+2c(q_1^2+q_1q_2+q_2^2+p_2^2)\bar{q}^2+2cq_1\big[(p_1+p_2)\bar{p}\big]\bar{q}\,dx  \non\\
+2L_4\int&\big(\bar{p}\partial_1\partial_1q_1+p_2\partial_1\partial_1\bar{q}+2\bar{q}\partial_1\partial_2q_1+2q_2\partial_1\partial_2\bar{q}-\bar{p}\partial_2\partial_2q_1
-p_2\partial_2\partial_2\bar{q} \big)\bar{q}\,dx \non\\
= I_1+\cdots+&I_6.
\end{align}
Note that $p_1, p_2, q_1, q_2, \bar{p}, \bar{q} \in  L^\infty(0,
\infty; H^1(\Omega))\cap L_{loc}^2(0, \infty; H^2(\Omega))\cap
L^\infty\big((0, \infty)\times \Omega\big)$, hence we know by Lemma
\ref{def of C1 and C2} that
\begin{align} I_1+&I_4 \non\\
\leq&C\big(\|\bar{p}\|_{L^4(\Omega)}+\|\bar{q}\|_{L^4(\Omega)}\big)\big(\|\nabla\bar{p}\|+\|\nabla\bar{q}\|\big)
\big(\|\nabla{Q}_1\|_{L^4(\Omega)}+\|\nabla{Q}_2\|_{L^4(\Omega)}\big)
\non\\
\leq&
C\big(\|\Delta{Q}_1\|^\frac12\|\nabla{Q}_1\|^\frac12+\|\nabla{Q}_1\|+\|\Delta{Q}_2\|^\frac12\|\nabla{Q}_2\|^\frac{1}{2}+\|\nabla{Q}_2\|\big)
\big(\|\bar{p}\|^\frac12+\|\bar{q}\|^\frac12
\big)\big(\|\nabla\bar{p}\|^\frac32+\|\nabla\bar{q}\|^\frac32\big)
\non\\
\leq&\frac{\zeta}{9}\big(\|\nabla\bar{p}\|^2+\|\nabla\bar{q}\|^2\big)
+C\big(\|\Delta{Q}_1\|^2+\|\Delta{Q}_2\|^2\big)\big(\|\bar{p}\|^2+\|\bar{q}\|^2\big),
\end{align}
\begin{eqnarray}
I_2+I_5 \leq C\big(\|\bar{p}\|^2+\|\bar{q}\|^2\big).
\end{eqnarray}
For $I_3$, integrating by parts gives
 \begin{align*}
 I_3 &= -2L_4\Big\{ 2\int\bar{p}\partial_1\bar{p}\partial_1p_1\,dx+\int\partial_1p_2\bar{p}\partial_1\bar{p}\,dx+\int{p}_2(\partial_1\bar{p})^2\,dx
 +2\int\bar{p}\partial_1\bar{q}\partial_2p_1\,dx \\
    &\qquad\qquad\quad+2\int\bar{q}\partial_1\bar{p}\partial_2p_1\,dx+2\int\bar{p}\partial_1q_2\partial_2\bar{p}\,dx+2\int{q}_2\partial_1\bar{p}\partial_2\bar{p}\,dx
    \\
  &\qquad\qquad\quad-2\int\bar{p}\partial_2\bar{p}\partial_2p_1\,dx-\int\partial_2p_2\bar{p}\partial_2\bar{p}\,dx-\int{p}_2(\partial_2\bar{p})^2\,dx   \Big\}\\
  &
  = I_{3a}+\cdots+I_{3j}.
  \end{align*}
Among all these $I_{3a}, \cdots, I_{3j}$, we may estimate
separately. First, by the assumption
$$
\|{Q}_1\|_{L^\infty((0,\infty)\times\Omega)}\leq\sqrt{2\eta_2}, \ \
\ \ \
\|{Q}_2\|_{L^\infty((0,\infty)\times\Omega)}\leq\sqrt{2\eta_2},
$$
 we see
 \begin{align}
I_{3c}+&I_{3g}+I_{3j}\non\\
&=2L_4\Big\{\int{p}_2(\partial_1\bar{p})^2\,dx+2\int{q}_2\partial_1\bar{p}\partial_2\bar{p}\,dx-\int{p}_2(\partial_2\bar{p})^2\,dx\Big\}
\non\\
&\leq
2|L_4|\Big\{\|p_2\|_{L^\infty(\Omega)}\|\partial_1\bar{p}\|^2+\|q_2\|_{L^\infty(\Omega)}\big(\|\partial_1\bar{p}\|^2
+\|\partial_2\bar{p}\|^2\big)+\|p_2\|_{L^\infty(\Omega)}\|\partial_2\bar{p}\|^2
\Big\} \non\\
&\leq 2|L_4|\|h_2\|_{L^\infty(\Omega)}\big\{\|\partial_1\bar{p}\|^2+\|\partial_1\bar{p}\|^2+\|\partial_2\bar{p}\|^2+\|\partial_2\bar{p}\|^2 \big\} \non\\
&\leq 4|L_4|\|h_2\|_{L^\infty(\Omega)}\|\nabla\bar{p}\|^2 \non\\
&\leq\frac{4\zeta}{1+4\sqrt{2}}\|\nabla\bar{p}\|^2 \non\\
&\leq\frac{2\zeta}{3}\|\nabla\bar{p}\|^2. \non
\end{align}
Here $h_2=\sqrt{p_2^2+q_2^2}$ is defined in the same way as
\eqref{def:h}, and we know from \eqref{def of eta 2} and Proposition
\ref{proposition on 2D lower bound} that
$\|h\|_{L^\infty(\Omega)}\leq \sqrt{\eta_2}\leq
\frac{\zeta}{(1+4\sqrt{2})|L_4|}$. Next, similar to the estimates
for $I_1$ and $I_4$, we have
\begin{align*}
I_{3a}+I_{3b}+I_{3d}&+I_{3e}+I_{3f}+I_{3h}+I_{3i}\\
\leq
&\frac{\zeta}{9}\big(\|\nabla\bar{p}\|^2+\|\nabla\bar{q}\|^2\big)
+C\big(\|\Delta{Q}_1\|^2+\|\Delta{Q}_2\|^2\|+1\big)\big(\|\bar{p}\|^2+\|\bar{q}\|^2\big).
\end{align*}
Therefore,
\begin{equation}
I_3 \leq
\frac{7\zeta}{9}\|\nabla\bar{p}\|^2+\frac{\zeta}{9}\|\nabla\bar{q}\|^2+C\big(\|\Delta{Q}_1\|^2+\|\Delta{Q}_2\|^2+1\big)\big(\|\bar{p}\|^2+\|\bar{q}\|^2\big).
\end{equation}
We control $I_6$ in a manner similar to $I_3$:
\begin{equation}
 I_6 \leq
\frac{\zeta}{9}\|\nabla\bar{p}\|^2+\frac{7\zeta}{9}\|\nabla\bar{q}\|^2+C\big(\|\Delta{Q}_1\|^2+\|\Delta{Q}_2\|^2+1\big)\big(\|\bar{p}\|^2+\|\bar{q}\|^2\big).
\end{equation}
Combining our estimates we have
\begin{equation}
\frac12\frac{d}{dt}\big(\|\bar{p}\|^2+\|\bar{q}\|^2\big) \leq
C\big(\|\Delta{Q}_1\|^2+\|\Delta{Q}_2\|^2+1\big)\big(\|\bar{p}\|^2+\|\bar{q}\|^2\big),
\ \ \ \forall t>0. \label{uniqueness inequality}
\end{equation}
Here $C$ is a positive constant that depends on $Q_0$, $\tilde{Q}$,
and the coefficients of the system. Using Proposition
\ref{proposition on 2D lower bound} and \eqref{eqnEE}, then a direct
application of Gronwall's inequality leads to \eqref{continuous
dependence equation}
\end{proof}

\section{Blow up for large  initial data}
\label{sec:blow-up}
In this section we aim to prove Theorem~\ref{thmBlowUp} by constructing (large enough) initial data for
which the solution of \eqref{eqnQExpansion} exhibits a finite
time blow-up of the $L^2$ norm.
For this purpose we use a hedgehog type ansatz
\begin{equation}\label{eqnHA}
Q_{ij}(t,x) = \theta(t, |x|) S_{ij}, \qquad\text{where } \ S_{ij} =
\Big( \frac{x_ix_j}{|x|^2}-\frac{\delta_{ij}}{2}\Big),\, i,j=1,2
\end{equation}
on the spherical domain $B_{R_1}(0)\setminus B_{R_0}(0)$.
Using rotational symmetry of the ansatz and domain, we reduce the evolution of $Q$ to a scalar one dimensional scalar PDE for $\theta$.
For this it suffices to only take boundary conditions for $\theta$.
It turns out that boundary conditions of the form
%
\begin{equation}
  \theta(t, R_0)=\theta(t, R_1)\ge 0, \ \ \forall t>0.
\label{homogeneous BC}
\end{equation}
are enough for our purposes.
The main result of this section shows that any solution to~\eqref{eqnQExpansion} of the form~\eqref{eqnHA} with boundary conditions~\eqref{homogeneous BC} and large enough initial data blows up in finite time.

We begin with an evolution equation for $\theta$.
\begin{lemma}\label{lmaTheta}
Let $Q$ be of the form~\eqref{eqnHA}. Then $Q$
 is a smooth solution of~\eqref{eqnQExpansion} if and only if $\theta$ is a smooth solution of

\begin{equation}\partial_t\theta=L_4\bigg(\frac{(\theta')^2}{2}+\frac{\theta\theta'}{r}+\theta\theta''+\frac{6\theta^2}{r^2}\bigg)
+\zeta\theta''+\frac{\zeta\theta'}{r}-\frac{4\zeta\theta}{r^2}-a\theta-\frac{c\theta^3}{2},
 \label{EqThetaReduced}\end{equation}
where $\zeta$ is defined in \eqref{zeta}.
\end{lemma}
\begin{remark}
  By the coercivity condition \eqref{eqn2DCoercivity} we know $\zeta > 0$.
\end{remark}

Postponing the proof of Lemma~\ref{lmaTheta} to
Appendix~\ref{sxnCalculations}, we prove Theorem~\ref{thmBlowUp}.

\begin{proof}[Proof of Theorem~\ref{thmBlowUp}]
  Let $\theta_-=-\min\{\theta,0\}$. Multiplying equation \eqref{EqThetaReduced} by $-\theta_- r$,
integrating over $[R_0, R_1]$ and integrating by parts gives
\begin{align}
\frac{1}{2} \frac{d}{dt}&\INT \theta_-^2 r \, dr \non\\
&= -L_4\INT \Big[ \frac{\big(\theta'_- \big)^2 \theta_-}{2}r +
\theta_-^2 \theta_-' + \frac{6\theta_-^3}{r}  \Big]\, dr-L_4\INT \theta_-^2 \theta_-'' r \, dr-\zeta\INT (\theta_-')^2r\,dr\non\\
&\qquad-\zeta\INT \theta'_-\theta_-\,dr+\zeta\INT
\theta'_-\theta_-\,dr
    -4\zeta\INT \frac{\theta_-^2}{r}\,dr-\INT\Big(a\theta_-^2+\frac{c}{2}\theta_-^4\Big)\,r\,dr\non\\
&= -L_4\INT \Big[ \frac{(\theta'_-)^2 \theta_-}{2} r +
\theta_-^2\theta_-' + \frac{6\theta_-^3}{r}  \Big] \, dr
+L_4\INT\Big[ \theta_-^2 \theta_-' + 2( \theta_-' )^2 \theta_- r \Big] \, dr\non\\
&\qquad-\zeta\INT (\theta_-')^2r\,dr-4\zeta\INT \frac{\theta_-^2}{r}\,dr-\INT \Big(a\theta_-^2+\frac{c}{2}\theta_-^4\Big)\,r\,dr\non\\
&= \frac{3L_4}{2} \INT ( \theta_-')^2 \theta_- r\,dr - 6L_4 \INT
\frac{\theta_-^3}{r}\,dr-\zeta\INT(\theta_-')^2r\,dr-4\zeta\INT\frac{\theta_-^2}{r}\,dr\non\\
&\qquad-\INT \Big(a\theta_-^2+\frac{c}{2}\theta_-^4\Big)r\,dr.
\label{EstThetaL2}
\end{align}
Next multiplying~\eqref{EqThetaReduced} by $-\partial_t\theta_-r$
and integrating over $[R_0,R_1]$, and integrating by parts wherever
necessary gives
\begin{align*}
0 &\leq \INT(\partial_t\theta_-\partial_t\theta_-)r\,dr \\
&=-L_4\INT
\partial_t\theta_-\Big[\frac{(\theta_-')^2}{2}+\frac{\theta_-\theta_-'}{r}
+\frac{6\theta_{-}^2}{r^2}\Big]r\,dr+L_4\INT \partial_t\theta_-'\theta_-\theta_-' r\,dr \\
&\qquad+L_4\INT \partial_t\theta_-(\theta_-')^2r\,dr+L_4\INT
\partial_t\theta_-\theta_-\theta_-'\,dr
-\zeta\INT\partial_t\theta_-'\theta_-'r\,dr \\
&\qquad-\zeta\INT\partial_t\theta_-\theta_-'\,dr+\zeta\INT\partial_t\theta_-\theta_-'\,dr-2\zeta\frac{d}{dt}\INT\frac{\theta_-^2}{r}\,dr
-\frac{d}{dt}\INT \Big(\frac{a\theta_-^2}{2}+\frac{c\theta_-^4}{8}\Big)\,r\,dr \non\\
&=\frac{d}{dt}\INT\left\{L_4\theta_-\Big[\frac{(\theta_-')^2}{2}-\frac{2\theta_-^2}{r^2}\Big]
{-}\zeta\Big[\frac{(\theta_-')^2}{2} {+}\frac{2\theta_-^2}{r^2}\Big]
{-}\Big(\frac{a}{2}\theta_-^2+\frac{c}{8}\theta_-^4\Big)\right\}r\,dr.
\end{align*}
Hence if we denote by
\begin{equation}
\mathcal{F}(t) \defeq
\INT\left\{L_4\theta_-\Big[\frac{(\theta_-')^2}{2}-\frac{2\theta_-^2}{r^2}\Big]
{-}\zeta\Big[\frac{(\theta_-')^2}{2} {+}\frac{2\theta_-^2}{r^2}\Big]
{-}\Big(\frac{a}{2}\theta_-^2+\frac{c}{8}\theta_-^4\Big)\right\}r\,dr,
\end{equation}
we have $\mathcal{F}(t){\geq}\mathcal{F}(0)$ and
\begin{align}
-2\zeta\INT (\theta_-')^2 r\,dr &\geq
4\mathcal{F}(0){-}\INT\Big(2L_4\theta_-\big[(\theta_-')^2-\frac{4\theta_-^2}{r^2}\big]
-\frac{8\zeta\theta_-^2}{r^2}\Big)r\,dr \non\\
&\qquad+\INT \Big(2a\theta_-^2+\frac{c}{2}\theta_-^4\Big)r\,dr
\label{EstGradientTheta}
\end{align}
We divide the argument into two cases: $L_4 < 0$ and $L_4 > 0$.
Suppose first $L_4 < 0$. Then $\zeta>0$ shows that
$-\zeta\INT(\theta_-')^2r\,dr\geq -2\zeta\INT(\theta_-')^2r\,dr$.
Using \eqref{EstGradientTheta} in \eqref{EstThetaL2}, we obtain:
\begin{align*}
\frac{1}{2}\frac{d}{dt}\INT \theta_-^2 &r \, dr  \\
\ge  \frac{3L_4}{2}&\INT \big( \theta_-' \big)^2 \theta_- r\,dr -
6L_4 \INT
\frac{\theta_-^3}{r}\,dr-4\zeta\INT\frac{\theta_-^2}{r}\,dr-\INT
\Big(a\theta_-^2+\frac{c}{2}\theta_-^4\Big)r\,dr
\non\\
+4&\mathcal{F}(0)
{-}\INT\Big(2L_4\theta_-\big[(\theta_-')^2-\frac{4\theta_-^2}{r^2}\big]-\frac{8\zeta\theta_-^2}{r^2}\Big)r\,dr+\INT
\Big(2a\theta_-^2+\frac{c}{2}\theta_-^4\Big)r\,dr,
\end{align*}
which becomes:
\begin{align}\label{BlowUpEqMinus}
\frac{1}{2} \frac{d}{dt}&\INT \theta_-^2 r \, dr \non\\
&\ge {-\frac{L_4}{2}}\INT (\theta_-')^2 \theta_- r\,dr +2L_4 \INT
\frac{\theta_-^3}{r}\,dr+4\zeta\INT\frac{\theta_-^2}{r}\,dr+4\F(0)+a\INT
\theta_-^2r\,dr \non\\
&\geq{-\frac{L_4}{2}}R_0\INT (\theta_-')^2 \theta_-\,dr
+\frac{2L_4}{R_0}\INT{\theta_-^3}\,dr+4\F(0)-|a|\INT
\theta_-^2r\,dr.
\end{align}
Using Poincar\'e's inequality, we get
\[
\int_{R_0}^{R_1} \left( \theta_-' \right)^2 \theta_-  \, dr \geq
\frac{4 }{9} \int_{R_0}^{R_1} \Big[ \big( \theta_-^{3/2}\big)'
\Big]^2 \, dr \geq \frac{4 \pi^2}{9(R_1 - R_0)^2} \int_{R_0}^{R_1}
\theta_-^3 \, dr.
\]
Therefore, if we choose $R_0, R_1$ so that
\begin{equation} \frac{R_0^2
\pi^2}{9(R_1 - R_0)^2}>1,
 \end{equation}
the inequality~\eqref{BlowUpEqMinus} reduces to
\begin{align}\label{Blow_up_equ}
\frac{1}{2}\frac{d}{dt}&\INT\theta_-^2 r \, dr \non\\
&\geq
-\frac{L_4}{2}\Big[\frac{4R_0\pi^2}{9(R_1-R_0)^2}-\frac{4}{R_0}
\Big]\INT \theta_-^3\,dr-|a|\INT
\theta_-^2rdr  +4\mathcal{F}(0) \non\\
&\geq M_0\Big(\INT \theta_-^2r\,dr
\Big)^{\frac32}-|a|\INT\theta_-^2r\,dr+4\mathcal{F}(0).
\end{align}
Here
$$
   M_0 \defeq -\frac{2L_4R_0}{\sqrt{R_1^4-R_0^4}}\Big[\frac{\pi^2}{9(R_1-R_0)^2}-\frac{1}{R_0^2}\Big].
$$

Consequently, if one assumes $\INT\theta_{0-}^2 r\,dr$ is suitably
large, then \eqref{Blow_up_equ} will force $\INT \theta_-^2 \, r \,
dr \to \infty$ in finite time, concluding the proof when $L_4 < 0$.
The above argument with $\theta_{-}$ replaced by $\theta_{+}$ will
handle the case when $L_4 > 0$.
\end{proof}

\begin{remark}
Our technique does not seem to have a straightforward extension to domains which are not radially symmetric.
In such domains, we do not know if a similar phenomenon occurs for large enough initial data.
\end{remark}

\section{The physicality preservation argument}
\label{sec:physicality}

Our aim in this section is to prove Proposition~\ref{prop:physicality}, showing that certain eigenvalue constraints (the so-called
 physicality constraints) are preserved by the evolution equation~\eqref{eqnQExpansion}.
This issue is more subtle than the preservation of
 the $L^\infty$ norm.

\begin{proof}[Proof of Proposition~\ref{prop:physicality}]
Under the assumption $L_2+L_3=L_4=0$ and in $d=2,3$
system~\eqref{eqnQExpansion} becomes
\begin{equation}\label{eq:QsystemSimple}
\frac{\partial Q_{ij}}{\partial
t}=2L_1\Delta{Q_{ij}}-aQ_{ij}+b\Big(Q_{ik}Q_{kj}-\frac{\tr(Q^2)}{d}\delta_{ij}\Big)-c
\tr(Q^2)Q_{ij},
\end{equation}
with $i,j=1,\dots,d$. Note that when $d=2$, the constant $2L_1$ is replaced
by $\zeta=2L_1+L_2+L_3>0$ in \eqref{eq:QsystemSimple}.
Thus the argument below is also valid even if $L_2+L_3\neq 0$.
For consistency, we only consider $L_2+L_3=0$.

The proof %
 will be done by
using a nonlinear Trotter product formula (see for instance Ch.
$15$, Section $5$ in \cite{T97}). To briefly describe the idea, let
us denote by $e^{2tL_1\Delta}R$ the solution of the heat equation in
the whole space, starting from initial data $R$ (where $R$ is
assumed to take values into the space of $d\times d$ matrices):
\begin{equation}
\left(e^{2tL_1\Delta}R\right)_{ij}(t,x)=\frac{1}{(4\pi
t)^{d/2}}\int_{\RR^d} e^{\frac{|x-y|^2}{8tL_1}}
R_{ij}(y)\,dy,\,\,\,\,\,\,\, i,j=1,\dots,d,
\end{equation}
and by $S(t,\bar S)\in\mcS^{(d)}$ the flow generated by the ODE part
of \eqref{eq:QsystemSimple} i.e. $S(t,\bar S)$ satisfies:
\begin{equation}\label{eq:Sdef}
 \left\{\begin{array}{ll} \frac{\partial}{\partial
t}S_{ij}(t,\bar
S)&=-aS_{ij}+b\Big(S_{ik}S_{kj}-\frac{\tr(S^2)}{d}\delta_{ij}\Big)-c
\tr(S^2)S_{ij}\\\
S(0,\bar S)_{ij}&=\bar S_{ij}\end{array}\right.
\end{equation}
with $i,j=1,\dots,d$.

Then the Trotter formula provides a way of expressing the solution
of \eqref{eq:QsystemSimple} as a limit of successive superpositions
of solutions of the heat equation part and the ODE part, namely by
denoting $Q(t,x)$ the solution of \eqref{eq:QsystemSimple} starting
from initial data $Q_0(x)$ we have, loosely speaking:
$$
  Q(t,x)=\lim_{n\to\infty}\left(e^{2TL_1/n\Delta}S(T/n,\cdot)\right)^nQ_0, \ \ \ \ \forall t\in [0,T]
$$

Let us  note now that a set of the form

$$\{Q\in\R^{d\times d},
Q=Q^t; \beta\le\lambda_i(Q)\le \gamma,\textrm{ for all eigenvalues
}\lambda_i(Q)\textrm{ of }Q\}$$ is convex (as the largest eigenvalue
is a convex function of the matrix, while the smallest eigenvalue is
a concave function, see for instance \cite{BV04}).

It is then clear that if we manage to show that both
$e^{2tL_1\Delta}$ and $S(t,\cdot)$ preserve the closed convex hull
of the range of the initial data then this will also hold for the
limit $Q(t,x)$. The arguments consist of three steps:

\noindent\textbf{Step 1}: The convex hull preservation under the
heat flow.

Denote
$$\Phi_n(y)\defeq\left\{\begin{array}{ll} (4\pi t)^{d/2}\big(\int_{B_n(0)}
e^{-\frac{|y|^2}{8tL_1}}\,dy\big)^{-1} e^{-\frac{|y|^2}{8tL_1}}& \textrm{ for }|y|\le n,\\
0 &\textrm{ for }|y|>n.\end{array}\right.$$ For any $f\in
L^1(\RR^d)$, we obtain that
\begin{equation}
\frac{1}{(4\pi t)^{d/2}}\int_{\RR^d} f(x-y)\Phi_n(y)\,dy\to
e^{2tL_1\Delta} f(x),
\end{equation}
pointwise as $n\to\infty$.

Now, let us observe that the measures $\mu_n(y)=\Phi_n(y)\,dy$
belong to the set $\mathcal{M}_{+1}(B_n(0))$ of regular Borel
probability measures supported on $B_n(0)$. The extremal set of the
convex set $\mathcal{M}_{+1}(B_n(0))$ consists of delta measures
$\delta_x$  with $x\in B_n(0)$ (where $\delta_x(E)=1$ if and only if
$x\in E$ for any Borel set $E\subset B_n(0)$; see for instance
\cite{BS11}, Ex. $8.16$, p. $129$). On the other hand, by
Krein-Milman theorem (see also Ch. $8$ in \cite{BS11}), we know that
$\mu_n$ can be written as a limit of convex combinations of
extremals in the weak-star topology  of $\mathcal{M}_{+1}(B_n(0))$
interpreted as a subset of the dual space
$\left[C(B_n(0))\right]^*$, i.e.
$$
  \sum_{j=1}^{J(k)} \theta_j^k \delta_{x_j^k} \stackrel{\star}{\rightharpoonup}\mu_n  \ \ \  \textrm{ as} \ \ k\to\infty,
$$
with the convexity condition
$$
  \sum_{j=1}^{J(k)}\theta_j^k=1,
$$
where $\theta_j^k\ge 0, \forall 1\le j\le J(k),k\in\NN.$ Therefore,
for any $x\in \RR^d$ and $n$ large enough so that $|x|<n$, it holds
$$
  \lim_{k\to\infty}\sum_{j=1}^{J(k)} \theta_j^k f(x_j^k-x)=\int_{\RR^d} f(x-y)d\mu_n(y)\,dy.
$$
After passing to the limit $n\to\infty$, we henceforth get
$(e^{2tL_1\Delta}f)(x)$ is in the convex hull of the image of the
initial data $f$.

\noindent\textbf{Step 2:} The physicality preservation under the
ODE.

We divide the argument into two cases.

\noindent\textbf{The $2D$ case:} We consider the ODE:
\begin{equation}\label{ode2D:abstract}
\frac{d}{dt}Q=-\frac{\partial f_B}{\partial
Q}+\frac{1}{2}\tr\left(\frac{\partial f_B}{\partial Q}\right)\Id,
\end{equation}
for $Q$ denoting $2\times 2$ matrices, where we use the standard
bulk term:
\begin{equation}\label{bulk2D}
f_B(Q)=\frac{a}{2}\tr(Q^2)-\frac{b}{3}\tr(Q^3)+\frac{c}{4}\left(\tr(Q^2)\right)^2.
\end{equation}
Taking into account the specific form \eqref{bulk} of $f_B$, the
equation \eqref{ode:abstract} becomes:
\begin{equation}\label{ode2D:explicit}
\frac{d}{dt}Q=-aQ+b\Big(Q^2-\frac{1}{2}\tr(Q^2)\Id\Big)-cQ\tr(Q^2).
\end{equation}
Multiplying the equation scalarly by $Q$, and using that $\tr(Q)=0$
and also the fact, specific to $2\times 2$ Q-tensors, that
$\tr(Q^3)=0$ we obtain:
\begin{equation}\label{traceode2D:brute}
 \frac12\frac{d}{dt}|Q|^2=-a|Q|^2-c|Q|^4.
\end{equation}
Let $g(|Q|)\defeq -a|Q|^2-c|Q|^4=-c|Q|^2(|Q|^2+\frac{a}{c})$. We
consider two possibilities:

{\it Case A: $a\ge 0$.}  Then $g(|Q|)<0$, for $|Q|\not=0$. Hence
\eqref{traceode2D:brute} implies $|Q(t)|^2\le |Q(0)|^2$.

{\it Case B: $a<0$.} Then
\begin{equation} \label{2dest:gneg}
 g(|Q|)<0, \ \ \ \textrm{for} \ |Q|^2>-\frac{a}{c}>0.
\end{equation}

We claim that
\begin{equation}\label{2Dest:preservation}
|Q(0)|\le \sqrt{-\frac{a}{c}}\Rightarrow |Q(t)|\le
\sqrt{-\frac{a}{c}}, \ \ \ \forall t>0.
\end{equation}

In order to prove the claim let us assume for contradiction that
there exists a $\varepsilon>0$ such that at some positive time
$|Q(t)|= \sqrt{-\frac{a}{c}}+\varepsilon$ and let us denote by $t_0$
the smallest such positive time. Then equation
\eqref{traceode2D:brute} together with \eqref{2dest:gneg} imply that
$\frac{d}{dt}|Q|^2<0$ hence there exists an earlier time
$t_{-1}<t_0$ so that $|Q(t_{-1})|=\sqrt{-\frac{a}{c}}+\varepsilon$
contradicting our hypothesis on $t_0$ and proving the claim
\eqref{2Dest:preservation}.

\noindent\textbf{The $3D$ case}: We consider the ODE:
\begin{equation}\label{ode:abstract}
\frac{d}{dt}Q=-\frac{\partial f_B}{\partial
Q}+\frac{1}{3}\tr\left(\frac{\partial f_B}{\partial Q}\right)\Id,
\end{equation}
where we use the standard bulk term:
\begin{equation}\label{bulk}
f_B(Q)=\frac{a}{2}\tr(Q^2)-\frac{b}{3}\tr(Q^3)+\frac{c}{4}\left(\tr(Q^2)\right)^2.
\end{equation}
Taking into account the specific form \eqref{bulk} of $f_B$, the
equation \eqref{ode:abstract} becomes:
\begin{equation}\label{ode:explicit}
\frac{d}{dt}Q=-aQ+b\Big(Q^2-\frac{1}{3}\tr(Q^2)\Id\Big)-cQ\tr(Q^2).
\end{equation}
Now take the scalar product of this equation with $Q$. (Recall,
scalar product of matrices $A, B$ is defined by $(A,B) \defeq
\tr(AB)$ and $|A|=\sqrt{\tr(A^2)}$.) Using additionally the fact
that $\tr(Q)=0$ gives
\begin{equation}\label{traceode:brute}
\frac12\frac{d}{dt}|Q|^2=-a|Q|^2+b\tr(Q^3)-c|Q|^4.
\end{equation}
We recall that (see for instance \cite{MZ10}) we have $
|\tr(Q^3)|\le\frac{|Q|^3}{\sqrt{6}}$ which used in
\eqref{traceode:brute} (under assumptions \eqref{AssumptionsBulk})
implies:
\begin{equation}\label{traceode:ineq}
\frac{d}{dt}|Q|^2\le-a|Q|^2+\frac{b}{\sqrt{6}}|Q|^3-c|Q|^4.
\end{equation}
Let us denote $h(Q) \defeq -a|Q|^2+\frac{b}{\sqrt{6}}|Q|^3-c|Q|^4$.
Then the roots of $\frac{h(Q)}{|Q|^2}$ are
$\sqrt{\frac{2}{3}}s_\pm$, with \begin{equation}\label{spm}
s_\pm=\frac{b\pm\sqrt{b^2-24ac}}{4c}. \
\end{equation}
Then
\begin{equation}\label{hQpositive}
h(|Q|)<0\textrm{ for }|Q|>\sqrt{\frac{2}{3}}s_+.
\end{equation}

Taking into account \eqref{traceode:ineq} we claim that, if we
denote by $Q_0$ the initial data of the ODE \eqref{ode:explicit}
\begin{equation}\label{boundpreservation}
|Q_0|^2\le \frac{2}{3}s_+^2\Rightarrow |Q(t)|^2\le \frac{2}{3}s_+^2,
\ \ \forall t>0.
\end{equation}

Indeed, if our claim were false, for any $\varepsilon>0$, let us
denote by $t_0(\varepsilon)$ the first time when $|Q|^2$ reaches the
value $\frac{2}{3}s_+^2+\varepsilon$, i.e.
$$
  |Q(t_0)|^2=\frac{2}{3}s_+^2+\varepsilon, \ \textrm{ and } \ |Q(t)|^2<\frac{2}{3}s_+^2+\varepsilon, \ \ \forall t<t_0.
$$

Then \eqref{hQpositive} and  \eqref{traceode:ineq} imply that
$\frac{d}{dt}|Q(t_0)|^2<0$. Hence there exists a time $\tilde
t_0<t_0$, such that $|Q(\tilde t_0)|>\frac{2}{3}s_+^2+\varepsilon$,
which contradicts our choice of $t_0$. Thus for $|Q_0|^2\le
\frac{2}{3}s_+^2$, the equation \eqref{ode:explicit} has a solution
that is bounded, and the right hand side of \eqref{ode:explicit} is
globally Lipschitz on the ball where the solution evolves. As a
consequence, we obtain that for $|Q_0|^2\le \frac{2}{3}s_+^2$, the
equation \eqref{ode:explicit} has a unique global solution evolving
with the property that $|Q(t)|^2\le \frac{2}{3}s_+^2$.

Let us consider now the system:
\begin{align}\label{eigensystem}
\frac{d\lambda_1}{dt}&=-\lambda_1\big[2c(\lambda_1^2+\lambda_2^2+\lambda_1\lambda_2)+a\big]+b\Big(\frac{\lambda_1^2}{3}
-\frac{2}{3}\lambda_2^2-\frac{2}{3}\lambda_1\lambda_2\Big),\nonumber\\
\frac{d\lambda_2}{dt}&=-\lambda_2\big[2c(\lambda_1^2+\lambda_2^2+\lambda_1\lambda_2)+a\big]
+b\Big(\frac{\lambda_2^2}{3}-\frac{2}{3}\lambda_1^2-\frac{2}{3}\lambda_1\lambda_2\Big).
\end{align}
The right hand side of the system is a locally Lipschitz function so
the system has a solution locally in time (in fact with some more
work global  in time and bounded, using arguments similar to the
ones before for the matrix system).

On the other hand, let us note now that if we take
$$
  Q_0=\left(\begin{array}{lll} \lambda_1^0 & 0 & 0 \\ 0 &\lambda_2^0 & 0\\ 0 & 0 & -\lambda_1^0-\lambda_2^0\end{array}\right),
$$
then
$$
  \bar Q(t)=\left(\begin{array}{lll} \lambda_1(t) & 0 & 0 \\ 0 &\lambda_2(t) & 0\\ 0 & 0 & -\lambda_1(t)-\lambda_2(t)\end{array}\right).
$$
Hence if $\lambda_1(t),\lambda_2(t)$ are solutions of
\eqref{eigensystem} with initial data $(\lambda_1^0,\lambda_2^0)$
then $\bar Q(t)$ is a solution of \eqref{ode:explicit} with initial
data $Q_0$. On the other hand, by uniqueness of solutions of
\eqref{ode:explicit}, it must be the only solution corresponding to
the diagonal initial data $Q_0$. Thus we have shown that a diagonal
initial data will generate a diagonal solution.

For an arbitrary, non-diagonal initial data $\tilde Q_0$, since
$\tilde Q_0$ is a symmetric matrix, there exists a matrix $R\in
O(3)$, such that
$$
  R\tilde
  Q_0R^t=\left(\begin{array}{lll}\tilde\lambda_1^0 & 0 & 0 \\ 0
  &\tilde\lambda_2^0 & 0\\ 0 & 0 &
   -\tilde\lambda_1^0-\tilde\lambda_2^0\end{array}\right),
$$
where $(\tilde \lambda_1^0,\tilde
\lambda_2^0,-\tilde\lambda_1^0-\tilde\lambda_2^0)$ are the
eigenvalues of $\tilde Q_0$. If $Q(t)$ is a solution of
\eqref{ode:explicit} with initial data $\tilde Q_0$, then
multiplying on the left by {\it the time independent matrix }$R$,
and on the right by {\it the time independent matrix} $R^t$, using
the fact that $RR^t=\Id$ (as $R\in O(3)$), we obtain the following
equation:
\begin{align}
\frac{d}{dt}RQ(t)R^t=-&aRQ(t)R^t+b\left(RQ(t)R^tRQ(t)R^t-\frac{1}{3}\tr (RQ(t)R^tRQ(t)R^t)\Id\right)\non\\
&-cRQ(t)R^t\tr\big(RQ(t)R^tRQ(t)R^t\big).
\end{align}
Hence if we denote by $M(t) \defeq RQ(t)R^t$, we conclude that $M$
satisfies equation \eqref{ode:explicit}
 with initial data
$$
   M_0 \defeq R\tilde Q_0R^t=\left(\begin{array}{lll}\tilde\lambda_1^0 & 0 & 0 \\ 0 &\tilde\lambda_2^0 & 0\\ 0 & 0 &
    -\tilde\lambda_1^0-\tilde\lambda_2^0\end{array}\right).
$$
Since the initial data is diagonal, we infer by previous arguments
that $M(t)$ is diagonal for all times and
$$
   M(t)=\left(\begin{array}{lll} \lambda_1(t) & 0 & 0 \\ 0 &\lambda_2(t) & 0\\ 0 & 0 &
    -\lambda_1(t)-\lambda_2(t)\end{array}\right),
$$
with $\lambda_1(t),\lambda_2(t)$ solutions of \eqref{eigensystem}
with initial data $(\tilde\lambda_1^0,\tilde\lambda_2^0)$. Thus we
obtain that
$$
  M(t)=RQ(t)R^t=\left(\begin{array}{lll} \lambda_1(t)
  & 0 & 0 \\ 0 &\lambda_2(t) & 0\\ 0 & 0 &
  -\lambda_1(t)-\lambda_2(t)\end{array}\right),
$$
hence
$$
  Q(t)=R^t\left(\begin{array}{lll} \lambda_1(t) & 0 & 0 \\ 0
  &\lambda_2(t) & 0\\ 0 & 0 &
  -\lambda_1(t)-\lambda_2(t)\end{array}\right)R.
$$
  This shows that we can reduce the study of the system \eqref{ode:explicit} with an
   arbitrary initial data to the study of the system \eqref{eigensystem}.

The bound \eqref{boundpreservation} expressed in terms of
eigenvalues $\lambda_1,\lambda_2$ becomes
\begin{equation}\label{boundpreservation+}
2\left[(\lambda_1^0)^2+(\lambda_2^0)^2+\lambda_1^0\lambda_2^0)\right]{\leq}\frac{2}{3}s_+^2,
\ \ \ \forall t\ge 0.
\end{equation}
Note that $\frac{3\lambda_i^2}{4}\le
\lambda_i^2+\mu^2+\lambda_i\mu$, hence the last bound implies
\begin{equation}\label{boundpreservation++}
2\left[(\lambda_1^0)^2+(\lambda_2^0)^2+\lambda_1^0\lambda_2^0)\right]\leq\frac{2}{3}s_+^2
\ \ \Rightarrow \ \ |\lambda_1(t)|,|\lambda_2(t)|\leq
\frac{2}{3}s_+, \ \ \ \forall t\ge 0.
\end{equation}

We consider now the difference $\lambda_1(t)-\lambda_2(t)$, and out
of inspection from the system \eqref{eigensystem} we see that it
satisfies an equation of the form:
$$
   \frac{d}{dt}(\lambda_1(t)-\lambda_2(t))=(\lambda_1(t)-\lambda_2(t))G(\lambda_1(t),\lambda_2(t)),
$$
for some function $G$. This shows that if $\lambda_1^0\le
\lambda_2^0$, then $\lambda_1(t)\le \lambda_2(t),\forall t>0$. We
assume without loss of generality that this is indeed the case.

We aim to show now that  $\lambda_1(0)\ge-\frac{s_+}{3}$ implies
$\lambda_1(t)\ge-\frac{s_+}{3}$  for all $t>0$. We assume for
contradiction that this is not the case and there exists  a first
time $t_0$ afterwhich $\lambda_1(t)+\frac{s_+}{3}$ becomes negative,
i.e. $\lambda_1(t_0)=-\frac{s_+}{3}$ and there exists a $\delta>0$
so that $\lambda_1(t)<-\frac{s+}{3}$ for $t\in(t_0,t_0+\delta)$. The
right hand side of equation \eqref{eigensystem} evaluated at $t_0$
becomes:
\begin{equation}\label{eq:sign-s3}
\frac{2}{3}(cs_+-b)(\lambda_2(t_0)+\frac{s_+}{3})(\lambda_2(t_0)-\frac{2s_+}{3}).
\end{equation}

Then equation \eqref{boundpreservation+} implies $\lambda_2(t_0)\in
[-\frac{s_+}{3},\frac{2s_+}{3}]$. If
$\lambda_2(t_0)\in\{-\frac{s_+}{3},\frac{2s_+}{3}\}$, then for all
$t>0$ we have
$\lambda_1(t)=-\frac{s_+}{3},\lambda_2(t)=\lambda_2(t_0)$, due to
the fact that the pairs
$(-\frac{s_+}{3},\frac{2s_+}{3}),(-\frac{s_+}{3},-\frac{s_+}{3})$
are stationary points of the system \eqref{eigensystem}. Thus we
assume without loss of generality that $\lambda_2(t_0)\in
(-\frac{s_+}{3},\frac{2s_+}{3})$ and henceforth, taking into account
assumption \eqref{restriction:a}, we infer that the expression in
\eqref{eq:sign-s3} is positive so $\frac{d\lambda_1}{dt}(t_0)>0$,
which contradicts our assumption that there exists a $\delta>0$ so
that $\lambda_1(t)<-\frac{s+}{3}$ for $t\in(t_0,t_0+\delta)$.

Thus we have shown that if $-\frac{s_+}{3}\le
\lambda_1^0\le\lambda_2^0\le\frac{2s_+}{3}$, then $\lambda_1(t)\in
[-\frac{s_+}{3},\frac{2s_+}{3}]$ for all $t>0$. The fact that
$\lambda_1(t)\le \lambda_2(t)$ for all times ensures
$-\frac{s+}{3}\le \lambda_2(t)$ for all times.

\noindent\textbf{Step 3}: The Trotter product formula

We use Proposition $5.3$ on p.$313$ in \cite{T97}. To this end, we
denote
$$
   V_n(t)\defeq
    e^{s\Delta}S(s,\cdot)\left(e^{2TL_1/n\Delta}S(T/n,\cdot)\right)^kQ_0,
$$
for $t=\frac{kT}{n}+s$ with $0\le s<\frac{T}{k}$. Then Proposition
$5.3$ ensures that we have:
\begin{equation}
 \|Q(t,\cdot)-V_n(t)\|_{H^k}\le C(\|Q_0\|_{H^k})n^{-\gamma},
\end{equation}
for $0<\gamma<1$, and all $t\in [0,T]$.
\end{proof}

\appendix\appendixpage
\section{Derivation of the gradient flow equation}
\label{sec:gradient flow derivation} \maketitle

Our goal in this subsection is to derive~\eqref{eqnQExpansion},
the equation for the gradient flow of $\E$.

\begin{proposition}
  The gradient flow defined by~\eqref{EquationQAbstract} satisfies~\eqref{eqnQExpansion}.
\end{proposition}

\begin{proof}
  Choosing a test function $\varphi \in C_c^\infty( \Omega, M^{d\times{d}}(\mathbb{R}) )$ and integrating by parts gives
\begin{align*}
\frac{d}{dt}\mathcal{E}&(Q+t\varphi)\Big|_{t=0} \\
&=\frac{d}{dt}\int_{\Omega}\mathcal{F}_{el}(Q+t\varphi)\,dx+\frac{d}{dt}\int_{\Omega}\mathcal{F}_{bulk}(Q+t\varphi)\,dx\\
&=\int_{\Omega}2L_1\partial_k\varphi_{ij}\partial_kQ_{ij}+L_2(\partial_j\varphi_{ik}\partial_kQ_{ij}+\partial_k\varphi_{ij}\partial_jQ_{ik})
+L_3(\partial_j\varphi_{ij}\partial_kQ_{ik}+\partial_k\varphi_{ik}\partial_jQ_{ij})\\
&\qquad\quad+L_4(\varphi_{lk}\partial_kQ_{ij}\partial_lQ_{ij}+Q_{lk}\partial_k\varphi_{ij}\partial_lQ_{ij}+Q_{lk}\partial_kQ_{ij}\partial_l\varphi_{ij})\,dx\\
&\qquad+\int_{\Omega}aQ_{ij}\varphi_{ij}-\frac{b}{3}\big(\varphi_{ik}Q_{kj}Q_{ji}+Q_{ik}\varphi_{kj}Q_{ji}+Q_{ik}Q_{kj}\varphi_{ji}\big)
+c\tr(Q^2)Q_{ij}\varphi_{ij}\,dx\\
&=\int_{\Omega}(-2L_1\Delta{Q_{ij}}-2L_2\partial_j\partial_kQ_{ik}-2L_3\partial_j\partial_kQ_{ik}
-2L_4\partial_lQ_{ij}\partial_kQ_{lk}
-2L_4\partial_l\partial_kQ_{ij}Q_{lk})\\
&\qquad\quad+L_4\partial_iQ_{kl}\partial_jQ_{kl})\varphi_{ij}\,dx+
\int_{\Omega}aQ_{ij}\varphi_{ij}-bQ_{jk}Q_{ki}\varphi_{ij}+c\tr(Q^2)Q_{ij}\varphi_{ij}\,dx.
\end{align*}
Since $\varphi$ is arbitrary this allows the identification
\begin{align*}
\left(\frac{\delta\mathcal{E}}{\delta Q}\right)_{ij}
&=-2L_1\Delta{Q_{ij}}+aQ_{ij}-bQ_{jk}Q_{ki}+c
\tr(Q^2)Q_{ij} \\
&\qquad-2(L_2+L_3)\partial_j\partial_kQ_{ik}-2L_4\partial_lQ_{ij}\partial_kQ_{lk}
-2L_4\partial_l\partial_kQ_{ij}Q_{lk}+L_4\partial_iQ_{kl}\partial_jQ_{kl}.
\end{align*}

Substituting this in~\eqref{EquationQAbstract} and choosing $\mu$ to
enforce the symmetry constraint $Q_{ij} = Q_{ji}$ forces
$$
  \mu_{ij}-\mu_{ji}=(L_2+L_3)\left(\partial_i\partial_kQ_{jk}-\partial_j\partial_kQ_{ik}\right).
$$
Similarly, choosing $\lambda$ to enforce the trace free constraint
$Q_{ii} = 0$ forces
$$
  \lambda=-\frac{b}{2}\tr(Q^2)-(L_2+L_3)\partial_l\partial_kQ_{lk}+\frac{L_4}{2}|\nabla{Q}|^2.
$$
Substituting $\lambda$, $\mu$ and $\delta \E / \delta Q$
in~\eqref{EquationQAbstract} immediately gives~\eqref{eqnQExpansion}.
\end{proof}

\section{The reduction of the Landau-de Gennes to Oseen-Frank in \texorpdfstring{$2D$}{2D}}\label{OF_LD}

Our goal in this appendix is to show that if $Q$ takes a special form, then the Landau-de~Gennes energy can be reduced to the Oseen Frank energy functional.
We recall that the $3D$ Oseen-Frank energy functional is
\begin{equation}
W=K_1(\mbox{div}\, n)^2+K_2|n\cdot \mbox{curl}\,n|^2+K_3|n\wedge
\mbox{curl}\,n|^2+(K_2+K_4)\big[\tr (\nabla n)^2-(\mbox{div}\,
n)^2\big].\label{OF}
\end{equation} where $K_i$ are elastic constants measuring the relative strength of the various types of spatial variations of the unit vectors $n\in \mathbb{S}^2$  (see \cite{oseen-frank}).
In 2D we clarify that for a vector function $n$ given by
$$
n=(n_1,n_2, 0),
$$
we have
$$
  \mbox{curl}\,n=(0, 0, \partial_1n_2-\partial_2n_1),
$$
and hence%
$$
   n\cdot\mbox{curl}\,{n}=0, \ \ |n\wedge\mbox{curl}\,n|^2=|\mbox{curl}\,n|^2.
$$
On the other hand, $n_1^2 + n_2^2 = 1$ implies
$$
   (n_1,n_2, 0) \cdot \partial_1 (n_1, n_2, 0) = 
   (n_1,n_2, 0) \cdot \partial_2 (n_1, n_2, 0) = 0.
$$
and hence
$
  (\partial_1n_1, \partial_1n_2, 0) = c(\partial_2n_1, \partial_2n_2, 0)
$
for some $c \in \R$.
Thus
$\partial_1n_1\partial_2n_2=\partial_2n_1\partial_1n_2$, which shows
$$
  \tr (\nabla n)^2=(\mbox{div}\, n)^2.
$$
Consequently, the Oseen-Frank energy in $2D$ reduces to
\begin{align}
W_{2D}&=K_1(\mbox{div}\,
n)^2+K_3|\mbox{curl}\,n|^2+(K_2+K_4)\big[\tr (\nabla
n)^2-(\mbox{div}\, n)^2\big] \non\\
&=K_1(\mbox{div}\, n)^2+K_3|\mbox{curl}\,n|^2 \label{OF2}.
\end{align}

If $Q$ takes the special form
$$
Q=s\Big(n\otimes{n}-\frac{\Id}{2}\Big),
$$
where $s$ is a constant, then the $2D$ Landau-de Gennes energy
functional reads
\begin{align}
\mathcal{E}(Q,
&\nabla{Q})\non\\
=&L_1|\nabla{Q}|^2+L_2\partial_jQ_{ik}\partial_kQ_{ij}+L_3\partial_jQ_{ij}\partial_kQ_{ik}
+L_4Q_{lk}\partial_lQ_{ij}\partial_kQ_{ij} \non\\
=&2L_1s^2\big[|\mbox{curl}\,n|^2+\tr(\nabla{n})^2\big]+L_2s^2\big[|\mbox{curl}\,n|^2+\tr(\nabla{n})^2\big]+L_3s^2\big[(\mbox{div}\,{n})^2+
|\mbox{curl}\,n|^2\big] \non\\
&+L_4s^3\big[|\mbox{curl}\,n|^2-\tr(\nabla{n})^2\big] \non\\
=&(2L_1+L_2)s^2\big[|\mbox{curl}\,n|^2+\tr(\nabla{n})^2\big]+L_3s^2\big[(\mbox{div}\,{n})^2+
|\mbox{curl}\,n|^2\big]+L_4s^3\big[|\mbox{curl}\,n|^2-\tr(\nabla{n})^2\big] \non\\
=&(\tilde{L}_1s^2+L_3s^2-L_4s^3)(\mbox{div}\,n)^2+(\tilde{L}_1s^2+L_3s^2+L_4s^3)|\mbox{curl}\,n|^2\non\\
&+(\tilde{L}_1s^2-L_4s^3)\big[\tr (\nabla n)^2-(\mbox{div}\,
n)^2\big] \non\\
=&(\tilde{L}_1s^2+L_3s^2-L_4s^3)(\mbox{div}\,n)^2+(\tilde{L}_1s^2+L_3s^2+L_4s^3)|\mbox{curl}\,n|^2.
\label{LD}
\end{align}
Here we denote
$$
   \tilde{L}_1=2L_1+L_2.
$$
We let
\begin{equation}
K_1=\tilde{L}_1s^2+L_3s^2-L_4s^3, \ \
K_3=\tilde{L}_1s^2+L_3s^2+L_4s^3, \label{map1}
\end{equation}
then $\mathcal{E}(Q, \nabla{Q})$ is reduced to $W_{2D}$. And
conversely, $\tilde{L}_1, L_3, L_4$ can be expressed in terms of
$K_i$ in the following way:
\begin{equation}
L_3s^2=K_1, \ \ 2L_4s^3=K_3-K_1, \ \
\tilde{L}_1s^2=\frac{K_3-K_1}{2}.
\end{equation}

\begin{remark}
  Note that if $L_4=0$, then $K_1 \equiv K_3$ in \eqref{map1},
which indicates that the Oseen-Frank energy \eqref{OF2} cannot be
completely recovered without $L_4$. Therefore, the cubic term is
necessary.
\end{remark}

\section{Energy coercivity in \texorpdfstring{$2D$}{2D}}\label{sxnCoercivity}
In this appendix we prove that the condition~\eqref{eqn2DCoercivity} (reproduced as~\eqref{coercivity in 2D} below) is equivalent to coercivity in two dimensions, and quantitatively gives the estimate~\eqref{eqn2DCoercivity1} (reproduced as~\eqref{positive energy} below).
As mentioned earlier, the three dimensional analog can be found in~\cite{LMT87,TE98}.
\begin{lemma}\label{lmaPositiveEnergy}
If $n=2$ and the elastic constants $L_1, L_2, L_3$ satisfy
\begin{equation}
L_1+L_2>0, \ \ L_1+L_3>0, \label{coercivity in 2D}
\end{equation}
then for all $x \in \Omega$ we have
\begin{equation}
\big(L_1|\nabla{Q}|^2+L_2\partial_jQ_{ik}\partial_kQ_{ij}
+L_3\partial_jQ_{ij}\partial_kQ_{ik}\big)(x) \geq \nu|\nabla{Q}|^2(x),
\label{positive energy}
\end{equation}
where
\begin{equation}\label{def of nu}
\nu\defeq \min\{L_1+L_2, L_1+L_3\}>0.
\end{equation}
\end{lemma}
\begin{proof}
Due to the special structure \eqref{Q in matrix form} of $Q$ in
$2D$, the elastic energy can be rewritten as
\begin{align}\label{2d coercive part}
\big(L_1|\nabla{Q}&|^2+L_2\partial_jQ_{ik}\partial_kQ_{ij}
+L_3\partial_jQ_{ij}\partial_kQ_{ik}\big)\\
&=(2L_1+L_2+L_3)\big(|\nabla{p}|^2+|\nabla{q}|^2\big)
+2(L_3-L_2)\partial_1p\partial_2q+2(L_2-L_3)\partial_2p\partial_1q \non\\
&
= \chi^T\mathcal{B}\chi, \non
\end{align}
where
$$
\chi=(\partial_1p, \partial_2p, \partial_1q, \partial_2q)^T \in
\mathbb{R}^4,
$$
and
$$
   \mathcal{B} = \left(
                     \begin{array}{cccc}
                       2L_1+L_2+L_3 & 0 & 0 & L_3-L_2 \\
                       0 & 2L_1+L_2+L_3 & L_2-L_3 & 0 \\
                       0 & L_2-L_3 & 2L_1+L_2+L_3 & 0 \\
                       L_3-L_2 & 0 & 0 & 2L_1+L_2+L_3 \\
                     \end{array}
                   \right) \in \mathbb{R}^{4\times 4}.
$$
By a direct calculation, we see that the eigenvalues of $\mathcal{B}$ are
$$
   \lambda_1=\lambda_2=2(L_1+L_2),
   \quad\text{and}\quad
   \lambda_3=\lambda_4=2(L_1+L_3).
$$
Consequently,
\begin{multline*}
\big(L_1|\nabla{Q}|^2+L_2\partial_jQ_{ik}\partial_kQ_{ij}
+L_3\partial_jQ_{ij}\partial_kQ_{ik}\big)\\
=\chi^T\mathcal{B}\chi\geq
\min\{\lambda_1,\lambda_2\}|\chi|^2=2\nu\big[|\nabla{p}|^2+|\nabla{q}|^2\big]=\nu|\nabla{Q}|^2
\end{multline*}
as desired.
\end{proof}

\section{Calculations for the hedgehog ansatz}\label{sxnCalculations}
In this section we prove Lemma~\ref{lmaTheta} deriving the
evolution of~$\theta$ that reduces the gradient flow dynamics in the
case of the Hedgehog ansatz.

\subsection{Calculations for Hedgehog type solutions: \texorpdfstring{$L_1$}{L1} and
\texorpdfstring{$L_4$}{L4} parts}

\noindent We begin by computing the first derivative of $Q_{ij}$ in
terms of $\theta$.
$$
Q_{ij,k}=\partial_kQ_{ij}=\theta'\frac{x_k}{|x|}\Big(\frac{x_ix_j}{|x|^2}-\frac{\delta_{ij}}{2}
\Big)+\theta\Big(\frac{\delta_{ik}x_j+\delta_{jk}x_i}{|x|^2}-\frac{2x_ix_jx_k}{|x|^4}
\Big).
$$

\noindent Next we compute the second derivative of $Q_{ij}$ in terms
of $\theta$.
\begin{align}
Q_{ij,kl}
 =\theta''
   &\frac{x_kx_l}{|x|^2}\Big(\frac{x_ix_j}{|x|^2}-\frac{\delta_{ij}}{2}\Big)
    +\theta'\Big(\frac{\delta_{kl}}{|x|}-\frac{x_kx_l}{|x|^3}\Big)\Big(\frac{x_ix_j}{|x|^2}-\frac{\delta_{ij}}{2}\Big)\non\\
&+\theta'\frac{x_k}{|x|}\Big(\frac{\delta_{il}x_j}{|x|^2}+\frac{\delta_{jl}x_i}{|x|^2}-\frac{2x_ix_jx_l}{|x|^4}\Big)
    +\theta'\frac{x_l}{|x|}\Big(\frac{\delta_{ik}x_j}{|x|^2}+\frac{\delta_{jk}x_i}{|x|^2}-\frac{2x_ix_jx_k}{|x|^4}\Big) \non \\
&+\theta\Big[\frac{\delta_{ik}\delta_{jl}}{|x|^2}-\frac{2\delta_{ik}x_jx_l}{|x|^4}
    +\frac{\delta_{il}\delta_{jk}}{|x|^2}-\frac{2\delta_{jk}x_lx_i}{|x|^4}
    \Big] \non\\
&-\theta\Big[\frac{2(\delta_{il}x_jx_k+\delta_{jl}x_ix_k+\delta_{kl}x_ix_j)}{|x|^4}-\frac{8x_ix_jx_kx_l}{|x|^6}\Big].\non
\end{align}
Thus for the term $2L_4Q_{ij,l} Q_{lk,k}$ in~\eqref{eqnQExpansion}, we have
\begin{align}
2L_4Q_{ij,l}Q_{lk,k} =2&L_4\left[\theta'
\frac{x_l}{|x|}\Big(\frac{x_ix_j}{|x|^2}-\frac{\delta_{ij}}{2}\Big)
+\theta\Big(\frac{\delta_{il}x_j}{|x|^2}+\frac{\delta_{jl}x_i}{|x|^2}-\frac{2x_ix_jx_l}{|x|^4}\Big)\right]\non\\
&\times\left[\theta'\frac{x_k}{|x|}\left(\frac{x_lx_k}{|x|^2}-\frac{\delta_{lk}}{2}\right)+
    \theta\Big(\frac{\delta_{lk}x_k}{|x|^2}+\frac{x_l\delta_{kk}}{|x|^2}-\frac{2x_lx_kx_k}{|x|^4}\Big)\right] \non\\
=2&L_4\left[\theta'
   \frac{x_l}{|x|}\Big(\frac{x_ix_j}{|x|^2}-\frac{\delta_{ij}}{2}\Big)
   +\theta\Big(\frac{\delta_{il}x_j}{|x|^2}+\frac{x_i\delta_{jl}}{|x|^2}-\frac{2x_ix_jx_l}{|x|^4}\Big)\right]\non\\
&\times\Big(\theta'\frac{x_l}{2|x|}+\theta\frac{x_l}{|x|^2}\Big) \non\\
=2&L_4\theta'\Big(\frac{\theta'}{2}+\frac{\theta}{|x|}\Big)\Big(\frac{x_ix_j}{|x|^2}-\frac{\delta_{ij}}{2}\Big).\non
\end{align}
For the term $2L_4Q_{kl}Q_{ij,kl}$, we get
\begin{align}
&2L_4Q_{kl}Q_{ij,kl} \non\\
&=2L_4\theta\Big(\frac{x_kx_l}{|x|^2}-\frac{\delta_{kl}}{2}\Big)\times
   \left[ \theta'' \frac{x_kx_l}{|x|^2} + \theta'\Big(\frac{\delta_{kl}}{|x|}-\frac{x_kx_l}{|x|^3}\Big) \right]
    \Big(\frac{x_ix_j}{|x|^2}-\frac{\delta_{ij}}{2}\Big) \non \\
&\qquad+2L_4\theta\Big(\frac{x_kx_l}{|x|^2}-\frac{\delta_{kl}}{2}\Big)\times
    \left[\theta'\frac{x_k}{|x|}\Big(\frac{\delta_{il}x_j}{|x|^2}+\frac{x_i\delta_{jl}}{|x|^2}-\frac{2x_ix_jx_l}{|x|^4}\Big)
    \right] \non\\
&\qquad+2L_4\theta\Big(\frac{x_kx_l}{|x|^2}-\frac{\delta_{kl}}{2}\Big)\times
   \left[\theta'\frac{x_l}{|x|}\Big(\frac{\delta_{ik}x_j}{|x|^2}+\frac{x_i\delta_{jk}}{|x|^2}-\frac{2x_ix_jx_k}{|x|^4}\Big)\right] \non \\
&\qquad+2L_4\theta\Big(\frac{x_kx_l}{|x|^2}-\frac{\delta_{kl}}{2}\Big)\times\theta\Big[\frac{\delta_{ik}\delta_{jl}}{|x|^2}-\frac{2\delta_{ik}x_jx_l}{|x|^4}
    +\frac{\delta_{il}\delta_{jk}}{|x|^2}-\frac{2x_lx_i\delta_{jk}}{|x|^4}\Big]\non\\
&\qquad-2L_4\theta\Big(\frac{x_kx_l}{|x|^2}-\frac{\delta_{kl}}{2}\Big)\times
    \theta\Big[\frac{2(\delta_{il}x_jx_k+x_i\delta_{jl}x_k+x_ix_j\delta_{kl})}{|x|^4}-\frac{8{x}_ix_jx_kx_l}{|x|^6}\Big]
    \non\\
&=L_4\Big(\theta\theta''-\frac{\theta\theta'}{|x|}+\frac{4\theta^2}{|x|^2}
   \Big)\Big(\frac{x_ix_j}{|x|^2}-\frac{\delta_{ij}}{2}\Big). \non
\end{align}
For $-L_4Q_{kl,i}Q_{kl,j}$, we have
\begin{align}
-L_4Q_{kl,i}Q_{kl,j} =
  -&L_4\left[\theta'\frac{x_i}{|x|}\Big(\frac{x_kx_l}{|x|^2}-\frac{\delta_{kl}}{2}\Big)
   +\theta\Big(\frac{\delta_{ki}x_l}{|x|^2}+\frac{x_k\delta_{il}}{|x|^2}-\frac{2x_kx_lx_i}{|x|^4}\Big)\right]\non\\
&\times\left[\theta'\frac{x_j}{|x|}\Big(\frac{x_kx_l}{|x|^2}-\frac{\delta_{kl}}{2}\Big)
    +\theta\Big(\frac{\delta_{kj}x_l}{|x|^2}+\frac{x_k\delta_{jl}}{|x|^2}-\frac{2x_kx_lx_j}{|x|^4}\Big)\right]\non\\
=-&L_4(\theta')^2\frac{x_ix_j}{2|x|^2}-\frac{2L_4\theta^2}{|x|^2}\Big(\delta_{ij}-\frac{x_ix_j}{|x|^2}\Big).\non
\end{align}
and \begin{equation}
\frac{L_4}{2}|\nabla{Q}|^2\delta_{ij}=L_4\Big[\frac{\theta^2}{|x|^2}+\frac{(\theta')^2}{4}\Big]\delta_{ij}.
\non \end{equation}
For terms related to $L_1$, we get
\begin{align}
\Delta{Q}_{ij} = &Q_{ij,kk}
\non\\
=&\theta''\Big(\frac{x_ix_j}{|x|^2}-\frac{\delta_{ij}}{2}\Big)
    +\theta'\frac{1}{|x|}\Big(\frac{x_ix_j}{|x|^2}-\frac{\delta_{ij}}{2}\Big)
    +2\theta'\frac{x_k}{|x|}\Big(\frac{\delta_{ik}x_j}{|x|^2}+\frac{x_i\delta_{jk}}{|x|^2}-\frac{2x_ix_jx_k}{|x|^4}\Big) \non \\
&+\theta\left[\frac{2\delta_{ik}\delta_{jk}}{|x|^2}-\frac{2\delta_{ik}x_jx_k}{|x|^4}
    -\frac{2x_kx_i\delta_{jk}}{|x|^4}
    -\frac{2(\delta_{ik}x_jx_k+\delta_{jk}x_ix_k+\delta_{kk}x_ix_j)}{|x|^4}+\frac{8x_ix_j}{|x|^4}\right]
    \non\\
= &\Big(\theta''+\frac{\theta'}{|x|}-\frac{4\theta}{|x|^2}
  \Big)\Big(\frac{x_ix_j}{|x|^2}-\frac{\delta_{ij}}{2}\Big).\non
\end{align}

\subsection{Terms related to \texorpdfstring{$L_2+L_3$}{L2+L3}}

\noindent There are two extra terms in this case, namely
$2(L_2+L_3)\partial_j\partial_kQ_{ik}$ and
$-(L_2+L_3)\partial_l\partial_kQ_{lk}\delta_{ij}$. For the former, we
calculate
\begin{align}
Q_{ik,kj}=\theta''&\frac{x_kx_j}{|x|^2}\Big(\frac{x_ix_k}{|x|^2}-\frac{\delta_{ik}}{2}\Big)
    +\theta'\Big(\frac{\delta_{kj}}{|x|}-\frac{x_kx_j}{|x|^3}\Big)\Big(\frac{x_ix_k}{|x|^2}-\frac{\delta_{ik}}{2}\Big) \non \\
&+\theta'\frac{x_k}{|x|}\Big(\frac{\delta_{ij}x_k}{|x|^2}+\frac{x_i\delta_{kj}}{|x|^2}-\frac{2x_ix_kx_j}{|x|^4}\Big)
    +\theta'\frac{x_j}{|x|}\Big(\frac{\delta_{ik}x_k}{|x|^2}+\frac{x_i\delta_{kk}}{|x|^2}-\frac{2x_ix_kx_k}{|x|^4}\Big) \non \\
&+\theta\left[\frac{\delta_{ik}\delta_{kj}}{|x|^2}-\frac{2\delta_{ik}x_kx_j}{|x|^4}
    +\frac{\delta_{ij}\delta_{kk}}{|x|^2}-\frac{2x_jx_i\delta_{kk}}{|x|^4}\right]\non\\
&-\theta\left[\frac{2(\delta_{ij}x_kx_k+\delta_{jk}x_ix_k+\delta_{kj}x_ix_k)}{|x|^4}-\frac{8x_ix_kx_kx_j}{|x|^6}\right]\non\\
=\theta''&\Big(\frac{x_ix_j}{|x|^2}-\frac{x_ix_j}{2|x|^2}\Big)+\theta'\Big(
    \frac{x_ix_j}{|x|^3}-\frac{x_ix_j}{|x|^3}-\frac{\delta_{ij}}{2|x|}+\frac{x_ix_j}{2|x|^3}
    \Big)\non\\
    &+\theta'\Big(\frac{\delta_{ij}}{|x|}+\frac{x_ix_j}{|x|^3}-\frac{2x_ix_j}{|x|^3}\Big)
    +\theta'\Big(\frac{x_ix_j}{|x|^3}+\frac{2x_ix_j}{|x|^3}-\frac{2x_ix_j}{|x|^3}\Big)\non\\
    &+\theta\Big(\frac{\delta_{ij}}{|x|^2}-\frac{2x_ix_j}{|x|^4}+\frac{2\delta_{ij}}{|x|^2}
    -\frac{4x_ix_j}{|x|^4}-\frac{2\delta_{ij}}{|x|^2}-\frac{4x_ix_j}{|x|^4}+\frac{8x_ix_j}{|x|^4}  \Big) \non\\
=\frac{\theta''}{2}&\frac{x_ix_j}{|x|^2}+\frac{\theta'}{2|x|}\Big(\frac{x_ix_j}{|x|^2}+\delta_{ij}\Big)
    +\frac{\theta}{|x|^2}\Big(\delta_{ij}-\frac{2x_ix_j}{|x|^2}
    \Big).\non
\end{align}
While for the latter, it holds
\begin{align}
Q_{lk,lk}
=\theta''&\frac{x_kx_l}{|x|^2}\Big(\frac{x_lx_k}{|x|^2}-\frac{\delta_{lk}}{2}\Big)
    +\theta'\Big(\frac{\delta_{kl}}{|x|}-\frac{x_kx_l}{|x|^3}\Big)\Big(\frac{x_lx_k}{|x|^2}-\frac{\delta_{lk}}{2}\Big) \non \\
&+\theta'\frac{x_k}{|x|}\Big(\frac{\delta_{ll}x_k}{|x|^2}+\frac{x_l\delta_{kl}}{|x|^2}-\frac{2x_lx_lx_k}{|x|^4}\Big)
    +\theta'\frac{x_l}{|x|}\Big(\frac{\delta_{lk}x_k}{|x|^2}+\frac{x_l\delta_{kk}}{|x|^2}-\frac{2x_lx_kx_k}{|x|^4}\Big) \non \\
&+\theta\left[\frac{\delta_{lk}\delta_{lk}}{|x|^2}-\frac{2\delta_{lk}x_kx_l}{|x|^4}
    +\frac{\delta_{ll}\delta_{kk}}{|x|^2}-\frac{2x_lx_l\delta_{kk}}{|x|^4}\right]\non\\
&-\theta\left[\frac{2(\delta_{ll}x_kx_k+\delta_{lk}x_lx_k+\delta_{lk}x_lx_k)}{|x|^4}-\frac{8x_lx_lx_kx_k}{|x|^6}\right]
    \non\\
=\frac{\theta''}{2}&+\frac{3}{2|x|}\theta'. \non
\end{align}
We conclude after putting them together that
 \begin{align*}
(L_2+&L_3)(Q_{ik,kj}+Q_{jk,ki})-(L_2+L_3)Q_{lk,lk}\delta_{ij}
\\
&=(L_2+L_3)\Big(\theta''+\frac{\theta'}{|x|}-\frac{4\theta}{|x|^2}\Big)\Big(\frac{x_ix_j}{|x|^2}-\frac{\delta_{ij}}{2}\Big).
\end{align*}
It is noted that
$$
   \Big(\frac{x_ix_j}{|x|^2}-\frac{\delta_{ij}}{2}\Big)\Big(\frac{x_ix_j}{|x|^2}-\frac{\delta_{ij}}{2}\Big)
    =S_{ij}S_{ij}=\frac12.
$$
Hence summing up the above calculations, then taking the inner
product with $S$, and denoting
$$
    \zeta = 2L_1+(L_2+L_3),
$$
we arrive at the following equation for the scalar unknown $\theta$ only:
\begin{gather}
\partial_t\theta=L_4\bigg(\frac{(\theta')^2}{2}+\frac{\theta\theta'}{r}+\theta\theta''+\frac{6\theta^2}{r^2}\bigg)
+\zeta\theta''+\frac{\zeta\theta'}{r}-\frac{4\zeta\theta}{r^2}-a\theta-\frac{c}{2}\theta^3.
 \non
\end{gather}


\section*{Acknowledgments}
We thank Xinfu Chen, David Kinderlehrer
and Robert L. Pego for helpful discussions.
We also thank  the anonymous referee for a very careful reading of this paper and his/her comments and suggestions.

\end{document}